\documentclass[11pt,english]{amsart}
\usepackage[T1]{fontenc}
\usepackage[latin1]{inputenc}
\usepackage{babel}
\usepackage{enumerate}
\usepackage{verbatim}
\usepackage{booktabs}
\usepackage{amstext}
\usepackage{amsthm}
\usepackage{amssymb}

\usepackage{fullpage}

\makeatletter

\providecommand{\tabularnewline}{\\}

\numberwithin{equation}{section}
\numberwithin{figure}{section}
\theoremstyle{plain}
\newtheorem{thm}{\protect\theoremname}
\theoremstyle{definition}
\newtheorem{defn}[thm]{\protect\definitionname}
\theoremstyle{remark}
\newtheorem{rem}[thm]{\protect\remarkname}
\theoremstyle{plain}
\newtheorem{prop}[thm]{\protect\propositionname}
\theoremstyle{plain}
\newtheorem{lem}[thm]{\protect\lemmaname}
\theoremstyle{plain}
\newtheorem{cor}[thm]{\protect\corollaryname}

\usepackage{tikz}
\usepackage{pgfplots}

\usepackage{babel}

\usepackage[unicode=true,
bookmarks=false,
breaklinks=false,pdfborder={0 0 1},pagebackref,colorlinks=false]
{hyperref}

\providecommand{\corollaryname}{Corollary}
\providecommand{\definitionname}{Definition}
\providecommand{\propositionname}{Proposition}
\providecommand{\remarkname}{Remark}
\providecommand{\theoremname}{Theorem}

\makeatother

\providecommand{\corollaryname}{Corollary}
\providecommand{\definitionname}{Definition}
\providecommand{\lemmaname}{Lemma}
\providecommand{\propositionname}{Proposition}
\providecommand{\remarkname}{Remark}
\providecommand{\theoremname}{Theorem}

\begin{document}


\addtolength{\textwidth}{0mm} \addtolength{\hoffset}{-0mm} 

\global\long\def\AA{\mathbb{A}}%
\global\long\def\CC{\mathbb{C}}%
\global\long\def\BB{\mathbb{B}}%
\global\long\def\PP{\mathbb{P}}%
\global\long\def\QQ{\mathbb{Q}}%
\global\long\def\RR{\mathbb{R}}%
\global\long\def\FF{\mathbb{F}}%
\global\long\def\kk{\mathbb{K}}%

\global\long\def\DD{\mathbb{D}}%
\global\long\def\NN{\mathbb{N}}%
\global\long\def\ZZ{\mathbb{Z}}%
\global\long\def\HH{\mathbb{H}}%
\global\long\def\Gal{{\rm Gal}}%
\global\long\def\GG{\mathbb{G}}%
\global\long\def\UU{\mathbb{U}}%

\global\long\def\bA{\mathbf{A}}%
\global\long\def\kP{\mathfrak{P}}%
\global\long\def\kQ{\mathfrak{q}}%
\global\long\def\ka{\mathfrak{a}}%
\global\long\def\kP{\mathfrak{p}}%
\global\long\def\kn{\mathfrak{n}}%
\global\long\def\km{\mathfrak{m}}%
\global\long\def\cA{\mathfrak{\mathcal{A}}}%
\global\long\def\cB{\mathfrak{\mathcal{B}}}%
\global\long\def\cC{\mathfrak{\mathcal{C}}}%
\global\long\def\cD{\mathcal{D}}%
\global\long\def\cH{\mathcal{H}}%
\global\long\def\cK{\mathcal{K}}%

\global\long\def\cF{\mathcal{F}}%
\global\long\def\cI{\mathfrak{\mathcal{I}}}%
\global\long\def\cJ{\mathcal{J}}%

\global\long\def\cL{\mathcal{L}}%
\global\long\def\cM{\mathcal{M}}%
\global\long\def\cN{\mathcal{N}}%
\global\long\def\cO{\mathcal{O}}%
\global\long\def\cP{\mathcal{P}}%
\global\long\def\cQ{\mathcal{Q}}%
\global\long\def\cR{\mathcal{R}}%
\global\long\def\cS{\mathcal{S}}%
\global\long\def\cT{\mathcal{T}}%
\global\long\def\cW{\mathcal{W}}%
\global\long\def\kBS{\mathfrak{B}_{6}}%
\global\long\def\kR{\mathfrak{R}}%
\global\long\def\kU{\mathfrak{U}}%
\global\long\def\kUn{\mathfrak{U}_{9}}%
\global\long\def\ksU{\mathfrak{U}_{7}}%
\global\long\def\a{\alpha}%
\global\long\def\b{\beta}%
\global\long\def\d{\delta}%
\global\long\def\D{\Delta}%
\global\long\def\L{\Lambda}%
\global\long\def\g{\gamma}%
\global\long\def\om{\omega}%
\global\long\def\G{\Gamma}%
\global\long\def\d{\delta}%
\global\long\def\D{\Delta}%
\global\long\def\e{\varepsilon}%

\global\long\def\k{\mathbb{k}}%
\global\long\def\l{\lambda}%
\global\long\def\m{\mu}%
\global\long\def\o{\omega}%
\global\long\def\p{\pi}%
\global\long\def\P{\Pi}%
\global\long\def\s{\sigma}%
\global\long\def\S{\Sigma}%
\global\long\def\t{\theta}%
\global\long\def\T{\Theta}%
\global\long\def\f{\varphi}%
\global\long\def\ze{\zeta}%

\global\long\def\deg{{\rm deg}}%
\global\long\def\det{{\rm det}}%

\global\long\def\Dem{Proof: }%
\global\long\def\ker{{\rm Ker}}%
\global\long\def\im{{\rm Im}}%
\global\long\def\rk{{\rm rk}}%
\global\long\def\car{{\rm car}}%
\global\long\def\fix{{\rm Fix( }}%
\global\long\def\card{{\rm Card }}%
\global\long\def\codim{{\rm codim}}%
\global\long\def\coker{{\rm Coker}}%

\global\long\def\pgcd{{\rm pgcd}}%
\global\long\def\ppcm{{\rm ppcm}}%
\global\long\def\la{\langle}%
\global\long\def\ra{\rangle}%

\global\long\def\Alb{{\rm Alb}}%
\global\long\def\Jac{{\rm Jac}}%
\global\long\def\Disc{{\rm Disc}}%
\global\long\def\Tr{{\rm Tr}}%
\global\long\def\Nr{{\rm Nr}}%

\global\long\def\NS{{\rm NS}}%
\global\long\def\Pic{{\rm Pic}}%

\global\long\def\Km{{\rm Km}}%
\global\long\def\rk{{\rm rk}}%
\global\long\def\Hom{{\rm Hom}}%
\global\long\def\End{{\rm End}}%
\global\long\def\aut{{\rm Aut}}%
\global\long\def\SSm{{\rm S}}%
\global\long\def\psl{{\rm PSL}}%
\global\long\def\cu{{\rm (-2)}}%
\global\long\def\mod{{\rm \,mod\,}}%
\global\long\def\cros{{\rm Cross}}%
\global\long\def\nt{z_{o}}%
\global\long\def\co{\mathfrak{\mathcal{C}}_{0}}%
\global\long\def\ldt{\Lambda_{\{2\},\{3\}}}%
\global\long\def\ltd{\Lambda_{\{3\},\{2\}}}%
\global\long\def\lldt{\lambda_{\{2\},\{3\}}}%

\global\long\def\ldq{\Lambda_{\{2\},\{4\}}}%
\global\long\def\lldq{\lambda_{\{2\},\{4\}}}%

\title[Regular polygons and line operators]{Regular polygons, line operators, and elliptic modular surfaces as
realization spaces of matroids}
\author{Lukas~K\"uhne}
\address{Lukas K\"uhne, Universit\"at Bielefeld, Fakult\"at f\"ur Mathematik, Bielefeld,
Germany}
\email{lkuehne@math.uni-bielefeld.de}
\author{Xavier Roulleau}
\address{Xavier Roulleau, Universit\'e d'Angers, CNRS, LAREMA, SFR MATHSTIC,
F-49000 Angers, France}
\email{xavier.roulleau@univ-angers.fr}
\begin{abstract}
For an integer $n\geq 7$, we investigate the matroid realization space of a specific deformation
of the regular $n$-gon along with its lines of symmetry. It turns
out that this particular realization space is birational to the
elliptic modular surface $\Xi_{1}(n)$ over the modular curve $X_{1}(n)$.

In this way, we obtain a model of $\Xi_{1}(n)$ defined over the rational
numbers. Furthermore, a natural geometric operator acts on these matroid
realizations. On the elliptic modular surface, this operator corresponds
to the multiplication by $-2$ on the elliptic curves. This provides
a new geometric approach to computing multiplication by $-2$ on elliptic
curves. 
\end{abstract}

\maketitle

\section{Introduction}

The starting point of the present paper was the search for new interesting
line arrangements,~i.e., finite union of lines in the projective
plane, by using certain operators $\L$ (respectively\ $\Psi$) acting
on line (respectively\
point) arrangements introduced in \cite{OSO}. These operators led
us to discover line arrangements related to elliptic modular surfaces,
as explained below.

The operators are defined as follows: Let $\mathfrak{m,n}$ be two
sets of integers $\geq2$. For a given line arrangement $\cC=\ell_{1}+\dots+\ell_{s}$
in $\PP^{2}$, we denote by $\cP_{\mathfrak{m}}(\cC)$ the (possibly
empty) union of the $m$-points of $\cC$, for $m\in\mathfrak{m}$,
where an $m$-point is a point where exactly $m$ lines of $\cC$
intersect. For a given point arrangement $\cP$, i.e. a finite set
of points, let $\cL_{\mathfrak{n}}(\cP)$ denote the union of $n$-rich
lines, for $n\in\mathfrak{n}$, where an $n$-rich line is a line
containing exactly $n$ points of $\cP$. The operator $\L_{\mathfrak{m,n}}$
is defined by $\L_{\mathfrak{m,n}}=\cL_{\mathfrak{n}}\circ\cP_{\mathfrak{m}}$.
For example, $\L_{\{2\},\{k\}}(\cC)$ is the union of lines containing
exactly $k$ double points of $\cC$. Similarly, the operator $\Psi_{\mathfrak{m,n}}$,
which acts on point arrangements, is defined by $\Psi_{\mathfrak{m,n}}=\cP_{\mathfrak{n}}\circ\cL_{\mathfrak{m}}$.
Once a polarization on $\PP^{2}$ is fixed, we also use the dual operator
$\cD$, which maps a line arrangement to a point arrangement and vice
versa. These operators satisfy the relation $\Psi_{\mathfrak{m,n}}=\cD\circ\L_{\mathfrak{m,n}}\circ\cD$.

In \cite{RoulleauUna}, a family $\kU$ of arrangements of $6$ lines
is described. These line arrangements have the remarkable property
that for a generic line arrangement $\cC$ in $\kU$, the line arrangement
$\ldt(\cC)$ is again in $\kU$. The singularities of $\cC\in\kU$
are only double points; the operator $\ldt$ acts as a degree $2$
map on the one dimensional parameter space of such arrangements, and
the periodic points of $\kU$ under the action of $\ldt$ are strongly
related to Ceva line arrangements, which are prominent examples of
line arrangements.

Finding other families $\kU'$ of line arrangements together with
an action of operators $\L$ is therefore a natural question. We found
an infinite family of such examples, which we describe as follows:

For $n\geq3$, let $\cP_{n}$ be the polygonal line arrangement i.e.
the union $\cP_{n}=\cC_{0}^{r}\cup\cC_{1}^{r}$ of the regular $n$-gon
$\cC_{0}^{r}$ and its $n$ lines of symmetries $\cC_{1}^{r}$.
For $n\geq7$, there exists an operator $\L$ that depends on $n$, see Equation~\eqref{eq:lambda},
 such that $\cC_{1}^{r}=\L(\cC_{0}^{r})$;
for example when $n=2k+1$, we use $\L=\L_{\{2\},\{k\}}$. The
regular $n$-gon $\cC_{0}^{r}$ has $\tfrac{n(n-1)}{2}$ double points,
which become $\tfrac{n(n-1)}{2}$ triple points on the union $\cP_{n}=\cC_{0}^{r}\cup\cC_{1}^{r}$.
Furthermore, $\cC_{1}^{r}$ has a unique singular point, the center
of the regular $n$-gon (and $\L(\cC_{1}^{r})=\emptyset$).

Consider a line arrangement $\cC_{0}\cup\cC_{1}$ of $n+n$ lines,
which has properties close to $\cP_{n}=\cC_{0}^{r}\cup\cC_{1}^{r}$
in the following sense:\\
i) The line arrangement $\cC_{0}$ is the union of $n$ lines with
$\tfrac{n(n-1)}{2}$ double points, \\
ii) The incidences, i.e., which lines meet in intersections points of higher multiplicities, between the lines in $\cC_{0}$ and the lines in
$\cC_{1}$ are the same as those between the lines in $\cC_{0}^{r}$
and $\cC_{1}^{r}$, so that $\L(\cC_{0})=\cC_{1}$. \\
Then the union $\cC_{0}\cup\cC_{1}$ has again $\tfrac{n(n-1)}{2}$
triple points. However, contrary to the case of the regular $n$-gon,
we do not impose that the $n$ lines of $\cC_{1}$ meet at a unique
point (which would rigidify the configuration). Instead we require
that $\cC_{1}$ has $\tfrac{n(n-1)}{2}$ double points, as $\cC_{0}$;
see Figure \ref{fig:MATROID} for the case $n=7$. Note that since
by construction $\cC_{1}=\L(\cC_{0})$, we will often identify the
line arrangements $\cC_{0}\cup\cC_{1}$ and $\cC_{0}$.

We show that the line arrangement $\cC_{0}\cup\cC_{1}$ has in fact
a natural labelling so that one may define the matroid $M_{n}$ associated
to a line arrangement $\cC_{0}\cup\cC_{1}$: this is the combinatorial
data describing how lines meet. The line arrangement $\cC_{0}\cup\cC_{1}$
is then said a realization of $M_{n}$. For a matroid $M$, if $\cC=(\ell_{1},\dots,\ell_{s})$
is a realization of $M$ and $\g$ is a projective transformation,
then $\g\cC=(\g\ell_{1},\dots,\g\ell_{s})$ is also a realization
of $M$. For any matroid $M$, there exists a parameter space $\cS(M)$
(respectively $\cR(M)$) of realizations of $M$, (respectively of
realizations of $M$ modulo projective transformations).
Both of these spaces are affine schemes.
 The scheme
$\cR(M)$ is called the realization space of $M$. Note that the actions
on line arrangements of the operators $\L_{\mathfrak{m,n}}$ and of
the projective transformations commute. Thus if $[\cC]$ denotes the
orbit of a line arrangement $\cC$ under $\text{PGL}_{3}$, the orbit
$\L_{\mathfrak{m,n}}([\cC]):=[\L_{\mathfrak{m,n}}(\cC)]$ is well-defined. 

The following result holds over an algebraically closed field of characteristic
$0$: 
\begin{thm}
\label{thm:MAIN1}Suppose that $n\geq7$. The realization space $\cR_{n}=\cR(M_{n})$
is two dimensional and irreducible. If $\cC_{0}\cup\cC_{1}$ is a
generic realization of $M_{n}$, then $\cC_{2}=\L(\cC_{1})$ is an
arrangement of $n$ lines, moreover $\cC_{2}$ can be labeled so that
$\cC_{1}\cup\cC_{2}$ is a realization of $M_{n}$. 
\end{thm}

We also discuss the case of positive characteristics, and we expect that the same results should be true, at least in characteristic coprime to $n$.
Here generic means that it is a generic point in the parameter space
$\cS(M)$, that is a point avoiding a finite set of hypersurfaces.
In this paper, we also discuss the case of positive characteristic,
for which some of the results of Theorem \ref{thm:MAIN1} still hold.

Let us now explain how the realization space $\cR_{n}$ is related
to elliptic modular surfaces. Recall that the modular curve $X_{1}(n)$
(for $n\geq3$) parametrizes up to isomorphisms pairs $(E,t)$ of
an elliptic curve $E$ with a point $t$ of order $n$. These curves
are fine moduli spaces. They have been studied e.g. by Deligne-Rapoport
\cite{Deligne-Rapoport}, Katz-Mazur \cite{Katz-Mazur} and Conrad
\cite{Conrad}, and are prominent objects in arithmetic geometry,
see e.g.\ \cite[App. C, Section 13]{Silverman}. The modular surface
$\Xi_{1}(n)$ is a smooth elliptic surface which is the universal
space over the modular curve $X_{1}(n)$. Shioda \cite{Shioda} studied
it by using analytic uniformization: it is a compactification of the
quotient of $\HH\times\CC$ by the action of a group $\G_{1}(n)\rtimes\ZZ^{2}$,
for the modular subgroup $\G_{1}(n)$ of $\text{SL}_{2}(\ZZ)$, where
$\HH$ is the upper half plane. Alternatively one may view $\Xi_{1}(n)$
as a (compactification of the) parameter space of triples $(E,p,t)$
where $E$ is an elliptic curve (with neutral element $O$), $p$ a point on $E$ and $t$ a generator
of a cyclic $n$-torsion subgroup of~$E$. The elliptic fibration
$\Xi_{1}(n)\to X_{1}(n)$ is the map $(E,p,t)\mapsto(E,t)$. There
is a natural multiplication by $m\in\ZZ$ map, which is a rational map on the elliptic
surface $\Xi_{1}(n)$, denoted by $[m]$. For a triple $\varphi=(E,p,t)\in\Xi_{1}(n)$,
let us choose a model of $E$ as a smooth cubic with a flex at $O=nt$,
so that one may define the labeled line arrangement 
\[
[\varphi]=\cD((p+kt)_{k\in\ZZ/n\ZZ}),
\]
which is, modulo projective automorphisms of the plane, independent
of the choice of such cubic model of $E$ (here $\cD$ is the dual
operator). Theorem \ref{thm:MAIN1} is a consequence of Theorem \ref{thm:Main-X}
below, which gives a link between the surface $\Xi_{1}(n)$ and the
realization space $\cR_{n}$. We work over an algebraically closed
field of characteristic $0$: 
\begin{thm}
\label{thm:Main-X}For $n\geq7$, the map $\psi:\varphi\mapsto[\varphi]\cup \Lambda([\varphi])$
is a degree $9$ rational map from $\Xi_{1}(n)$ to $\cR_{n}$ and
the following diagram of rational maps commutes: 
\[
\begin{array}{ccc}
\Xi_{1}(n) & \stackrel{[-2]}{\to} & \Xi_{1}(n)\\
\psi\downarrow &  & \psi\downarrow\\
{}
\cR_{n} & \stackrel{\L}{\to} & \cR_{n}
\end{array}.
\]
The map $\psi$ induces a birational map $\Xi_{1}(n)/K({3})\simeq\Xi_{1}(n)\to\mathcal{R}_{n}$,
where $K({3})$ is the kernel of the multiplication by $3$ map on
the elliptic surface $\Xi_{1}(n)$. The degree of the map $\L$ is
$4$.
\end{thm}

Here $\L$ is the line operator mentioned above that is related to the regular $n$-gon, see Equation~\eqref{eq:lambda}.
Note that, unlike Shioda's construction of $\Xi_{1}(n)$, which is
by analytic uniformization, the schemes $\cR_{n}$ are naturally defined
over $\QQ$ (and even over $\ZZ$), since these are realization spaces
of matroids. In \cite{Chai}, Chai and Faltings constructed
the universal elliptic surface 
as well as its compactifications over $\mathbb{Z}$. 
At least for $n=7$, that model are not smooth over $\mathbb{Z}$, since $\Xi_1(n)$ is a K3 surface and  there is no K3 surface over $\mathbb{Z}$ by \cite{Abrashkin} and \cite{Fontaine}. 
We note also that the realization spaces $\cR_{n}$ are affine
schemes, see e.g. \cite{Oscar}.  
The surface $\Xi_1(n)$ is the unique minimal smooth 
compactification of the quasiprojective variety parametrizing triples
$(E, p, t)$ as above. We use the same notation for these two surfaces. 
For example, this is not a problem in the above diagram, 
since the maps $\Psi,\,\Lambda$ and $[-2]$ are rational.

One can also reformulate Theorem \ref{thm:Main-X} in terms of point
arrangements instead of line arrangements, by associating to a triple
$(E,p,t)$, the labeled point arrangement $\cP=(p+kt)_{k\in\ZZ/n\ZZ}$
and by using the point operator $\Psi=\cD\circ\L\circ\cD$. By doing
so, one obtains a geometric (and algorithmic) way to compute the multiplication
by $-2$ of certain points of a cubic curve, without needing to take
the tangents to the points. What is required is the computation of
the intersection points of the lines linking the points in $\cP$.
For $n\in\{7,8\}$, we describe such sets $\cP$ in \cite{KR2}.

Let us now describe the structure of this paper, and further results
obtained: In Section \ref{sec:prelims}, we review results regarding
the operators $\L,\Psi$, the matroids, and their realization spaces.
Section~\ref{sec:Elliptic-modular-surfaces} is devoted to the proof
of both Theorem~\ref{thm:MAIN1} and \ref{thm:Main-X}.
We start by describing the matroids $M_{n}$ associated with the
regular $n$-gon and its lines of symmetry in Subsection~\ref{subsec:Description-Matro}.
Subsequently, in Subsection~\ref{subsec:A-point-realization} we prove
that a generic realization of $M_n$ has a preimage in $\Xi_1(n)$ by
showing that there exists a cubic curve that
contains the dual points of a given realization.
Conversely, we
show in the Subsections~\ref{subsec:Labelled-arrangements-of} and \ref{subsec:Converse}
that the generic points on the elliptic modular surface $\Xi_{1}(n)$,
yield realizations of the matroid $M_{n}$.
In the Subsections  \ref{subsec:A-point-realization}, \ref{subsec:Labelled-arrangements-of} and
\ref{subsec:Converse} we also
discuss the case of the fields of positive characteristic. In Section
\ref{sec:Further-constructions-and}, we treat the cases of realizations
of $M_{n}$ obtained by using the singular cubic curves, and we study
some periodic line arrangements under the action of $\L$. In Sections
\ref{sec:The-pentagon} and \ref{sec:The-hexagon}, we generalize
our constructions and results to the modular surfaces $\Xi_{1}(5)$
and $\Xi_{1}(6)$. The limit case $n=5$ is of interest because it
is especially simple: we describe a combinatorial-geometric point
operator $\Psi$ such that for any arrangement $\cP$ of $5$ points
in generic
position, the successive images of~$\cP$ by the powers
of $\Psi$ are points on the same cubic curve. These points are also
the successive powers of the multiplication by $-2$ map on that curve.
For the cases $n=5$ and $n=6$, we also establish a connection between
our operators and the pentagram map, which is another type of operator
acting on line arrangements and has been intensively studied, see
\cite{SchwartzPent}.

Finally, let us note that the construction of the elliptic modular curves $X_{1}(n)$
as a realization space of a matroid is discussed in \cite{RoulleauCurves}.
The main result of this paper is that for $n\ge 10$, the elliptic modular curve $X_{1}(n)$ is birational to the realization space of the elliptic matroid $\mathcal{T}_n$, which is the rank $3$ matroid on the ground set $\{0,1,\dots, n-1\}$ with non-bases triples that sum to $0$ modulo $n$.
The proof methods are however fairly disjoint from the present one and rely on modular forms.

\subsection*{Acknowledgments}

Our work was initiated during the \textit{Workshop on Complex and
Symplectic Curve Configurations} held in Nantes, France, in December
2022. We would like to thank the organizers for stimulating and fruitful
discussions and providing excellent working conditions.
We are grateful to Ana Maria Botero, Bert van Geemen, Keiji Oguiso
and Will Sawin for inspiring discussions, and to Pierre Deligne for
pointing out an error in a first version of this paper.
We would also like to thank the anonymous referees for their insightful and constructive comments, 
which greatly helped to improve the clarity and overall quality of this manuscript.
We used the computer algebra systems \texttt{OSCAR} \cite{Oscar}
and \texttt{MAGMA} \cite{Magma}.

LK is supported by the Deutsche Forschungsgemeinschaft (DFG, German Research Foundation) -- SFB-TRR 358/1 2023 -- 491392403 and SPP 2458 -- 539866293.
 XR is supported by the French
Centre Henri Lebesgue ANR-11-LABX-0020-01.

\section{\label{sec:prelims}Preliminaries on operators and matroids}

\subsection{\label{subsec:Line-arrangement-and}The operators $\Lambda_{\mathfrak{n,m}}$
and $\Psi_{\mathfrak{n,m}}$}

A line arrangement $\cC=\ell_{1}+\dots+\ell_{n}$ is the union of
a finite number of distinct lines in the projective plane $\PP^{2}$
over some field $\kk$. A labeled line arrangement $\cC=(\ell_{1},\dots,\ell_{n})$
is a line arrangement with a fixed order of the lines. We sometimes
add a superscript~$^{\ell}$ (resp.\ $^{u}$) when we want to emphasize
that an arrangement or related objects has (resp.\ does not have)
a labeling.

If $\cC_{1}$ and $\cC_{2}$ are two labeled line arrangements without
common lines, the union $\cC_{1}\cup\cC_{2}$ is a labeled line arrangement,
and the order of the terms is important, as $\cC_{1}\cup\cC_{2}\neq\cC_{2}\cup\cC_{1}$
if the line arrangements are non-empty. 

Results in terms of points and lines yields a dual statement, obtained
by swapping the notions of points and lines, join with intersection,
and collinear with concurrent. Let us fix $\cD$ as the dual operator
between the plane $\PP^{2}$ and its dual $\check{\PP^{2}}$, which
to a line arrangement $\cC$ associates an arrangement of points,
namely the normals of the lines of $\cC$. Concretely, we fix coordinates
$x,y,z$ so that the line $\ell:\{ax+by+cz=0\}$ yields $\cD(\ell)=(a:b:c)$,
so that we will often identify the plane and its dual by using these
coordinates. 

By duality, the operators $\Lambda_{\mathfrak{n,m}}$ defined in
the introduction have their counterpart $\Psi_{\mathfrak{n,m}}=\cP_{\mathfrak{m}}\circ\cL_{\mathfrak{n}}$
on point arrangements~$\cP$,~i.e., finite set of points in $\PP^{2}$.
For example $\Psi_{\{2\},\{4\}}$ is the operator which to a point
arrangement $\cP$ returns the set of $4$-points in the union of
the lines that contain exactly two points of $\cP$. The operators
$\Lambda_{\mathfrak{m},\mathfrak{n}}$ and $\Psi_{\mathfrak{m},\mathfrak{n}}$
are related as follows: 
\[
\Psi_{\mathfrak{m},\mathfrak{n}}=\cD\circ\Lambda_{\mathfrak{m},\mathfrak{n}}\circ\cD.
\]

For a line arrangement $\cC$ and an integer $k\geq2$, we denote
by $t_{k}=t_{k}(\cC)=|\cP_{\{k\}}(\cC)|$ the number of $k$-points
of $\cC$.

\subsection{\label{sec:Matroids}Matroids}

A matroid is a fundamental and actively studied object in combinatorics.
Matroids generalize linear dependency in vector spaces as well as
many aspects of graph theory. See e.g.~\cite{Oxley} for a comprehensive
treatment of matroids. We briefly introduce the concepts from matroid
theory that will appear in this article.
\begin{defn}
A \emph{matroid} is a pair $M=(E,\cB)$, where $E$ is a finite set
of elements called atoms and $\cB$ is a nonempty collection of subsets
of $E$, called \emph{bases,} satisfying an exchange property reminiscent
of linear algebra: If $A$ and $B$ are distinct members of $\cB$
and $a\in A\setminus B$, then there exists $b\in B\setminus A$ such
that $(A\setminus\{a\})\cup\{b\}\in\cB$. 
\end{defn}

The prime examples of matroids arise by choosing a finite set of vectors
$E$ in a vector space and declaring the maximal linearly independent
subsets of $E$ as bases.

The basis exchange property already implies that all bases have the
same cardinality, say $r$, which is called the \emph{rank} of $(E,\cB)$.
The subsets of $E$ of order $r$ that are not basis are called \emph{non-bases}.

An isomorphism between the two matroids $M_{1}=(E_{1},\cB_{1}),$
$M_{2}=(E_{2},\cB_{2})$ is a bijection from $E_{1}$ to $E_{2}$
which maps the set of bases of $M_{1}$ bijectively to the set of
bases of $M_{2}$. We denote by $\aut(M)$ the \emph{automorphism
group} of the matroid $M$, i.e., the set of isomorphisms from $M$
to $M$.

As we will be only concerned with line (or point) arrangements in
$\PP^{2}$, we only consider matroids of rank $3$ from now on. If
the ground set $E$ is of order $m$ we identify $E$ with the set
$\{1,\dots,m\}$.

Matroids originated from the following kind of examples: If $\cC=(\ell_{1},\dots,\ell_{m})$
is a labeled line arrangement, the subsets $\{i,j,k\}\subseteq\{1,\dots,m\}$
such that the lines $\ell_{i},\ell_{j},\ell_{k}$ meet in three distinct
points are the bases of a matroid $M(\cC)$ over the set $\{1,\dots,m\}$.
We say that $M(\cC)$ is the matroid associated to $\cC$.

\subsection{The realization space of a matroid}

A \emph{realization} (over some field $\mathbb{K}$) of a matroid
$M=(E,\cB)$ of rank $3$ is the converse operation to the association
$\cC\to M(\cC)$. It is represented as a $3\times m$-matrix over
$\mathbb{K}$ with non-zero columns $C_{1},\dots,C_{m}$, considered
up to a multiplication by a scalar (thus as points in $\PP^{2}$).
A subset $\{i_{1},i_{2},i_{3}\}$ of $E$ of size $3$ is a basis
if and only if the $3\times3$ minor $|C_{i_{1}},C_{i_{2}},C_{i_{3}}|$
is nonzero. We denote by $\ell_{i}$ the line whose normal vector
is the point $C_{i}\in\PP^{2}$. In this context, a realization of
$M$ is a labeled line arrangement $\cC=(\ell_{1},\dots,\ell_{m})$,
where three lines $\ell_{i_{1}},\ell_{i_{2}},\ell_{i_{3}}$ meet at
a unique point if and only if $\{i_{1},i_{2},i_{3}\}$ is a non-basis.
We may also say that the point arrangement $\cP=C_{1},\dots,C_{m}$
is a realization of $M$. The points $C_{i},C_{j},C_{k}$ are collinear
if and only if $\{i,j,k\}$ is a non-basis.

If $\cC=(\ell_{1},\dots,\ell_{m})$ is a realization of $M$ and $\g\in PGL_{3}$,
then $(\g\ell_{1},\dots,\g\ell_{m})$, the image of $\cC$ by $\g$,
is another realization of $M$; we denote by $[\cC]$ the orbit of
$\cC$ under that action of $PGL_{3}$. The \emph{realization space}
$\cR(M)$ of realizations of $M$ parametrizes the orbits $[\cC]$
of realizations. That space $\cR(M)$ is an affine scheme constructed
from a $3\times m$ matrix with unknowns as entries and relations
the ideal generated by the minors of the non-bases, from which one
removes the zero loci of the minors of the bases. Moreover, since
each column $c$ is non-zero and considered up to multiplication by
$\CC^{*}$, one can suppose that one of the entries of $c$ is a $1$.
A more detailed introduction to these realization spaces together
with a description of a software package in \texttt{OSCAR} that can
compute the equations of these spaces is given in~\cite{Oscar}.

In this article, we always assume that each subset of three elements
of the first four atoms is a basis (otherwise, we replace $M$ by
a matroid isomorphic to it). Then in the realization space $\cR(M)$,
one can always map the first four vectors of $\cC\in[\cC]$ to the
canonical basis, so that each element $[\cC]$ of $\cR(M)$ has a
canonical representative, which we will identify with $[\cC]$.

Let $M_{k}=(E_{k},\cB_{k})$, $k\in \{1,2\}$ be two matroids with $E_{1}=E_{2}=\{1,\dots,m\}$.
If $\S:M_{1}\to M_{2}$ is an isomorphism, defined by a permutation
$\s$ of $\{1,\dots,m\}$ and if $\cC=(\ell_{1},\dots,\ell_{m})$
a realization of $M_{1}$, then $\S\cdot\cC:=(\ell_{\s1},\dots,\ell_{\s m})$
is a realization of $M_{2}$. Since the action of $PGL_{3}$ commutes
with the permutations of the lines, the map $\cC\to\S\cdot\cC$ induces
an isomorphism between the realization spaces $\cR(M_{1})\to\cR(M_{2})$,
in particular the group $\aut(M)$ acts on $\cR(M)$. That action
may be not faithful (for example, the matroid with $4$ atoms and
no non-basis has automorphism group $S_{4}$ but the realization space
is a point).

\section{\label{sec:Elliptic-modular-surfaces}Elliptic modular surfaces}

In this section, we describe a relationship between the realization
spaces of certain matroids and elliptic modular surfaces. We begin
by defining these matroids, which originate from regular polygons.

\subsection{\label{subsec:Description-Matro} Matroids from regular polygons}

\subsubsection{Odd number of sides}

Let $n=2k+1\geq5$ be an odd integer. Consider $\cC_{0}=(\ell_{1},\dots,\ell_{n})$
the lines of the regular $n$-gon in the real plane. We label the
lines $\ell_{j}$ anti-clockwise (see Figure \ref{fig:MATROID} for
the case $n=7$) and we consider the index $j$ of $\ell_{j}$ in
$\ZZ/n\ZZ$. For $i\neq j$ in $\ZZ/n\ZZ$, we denote by $p_{i,j}$
the intersection point of the lines $\ell{}_{i}$ and $\ell_{j}$.

The line arrangement $\cC_{0}$ has $\frac{n(n-1)}{2}$ double points
and these points have the property that for any $r\in\ZZ/n\ZZ$, the
$k=\tfrac{n-1}{2}$ double points $p_{i,j}$ ($i\neq j$) such that
\[
\,i+j+r=0\,\mod n
\]
are collinear. Let us denote by $\ell_{r}'$ the line containing these
$k$ points. Figure \ref{fig:CANO-label} illustrates this labeling
for the case $n=7$.

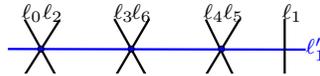
\begin{figure}[h]
\begin{centering}
\begin{tikzpicture}[scale=0.4]
\clip(-2.,1.2) rectangle (9.98,3.5);
\draw [line width=0.3mm] (-0.54,2.98)-- (0.58,0.96);
\draw [line width=0.3mm] (0.62,2.96)-- (-0.64,0.98);
\draw [line width=0.3mm] (2.44,3.)-- (3.56,1.);
\draw [line width=0.3mm] (3.6,3.)-- (2.42,1.);
\draw [line width=0.3mm] (5.48,2.98)-- (6.46,0.98);
\draw [line width=0.3mm] (6.6,2.98)-- (5.4,1.);
\draw [line width=0.3mm,color=blue] (-1.1,2.)-- (8.8,1.98);
\draw [line width=0.3mm] (8.1,3.)-- (8.1,1.);
\begin{scriptsize}
\draw [fill=blue] (0.,2.) circle (1mm);
\draw [fill=blue] (3.,2.) circle (1mm);
\draw [fill=blue] (6.,2.) circle (1mm);
\draw[color=black] (-0.3,3.2) node {$\ell_0$};
\draw[color=black] (0.4,3.2) node {$\ell_6$};
\draw[color=black] (2.75,3.2) node {$\ell_1$};
\draw[color=black] (3.35,3.2) node {$\ell_5$};
\draw[color=black] (5.74,3.2) node {$\ell_2$};
\draw[color=black] (6.4,3.2) node {$\ell_4$};
\draw[color=blue] (9.14,2.) node {$\ell_1'$};
\draw[color=black] (8.35,3.2) node {$\ell_3$};
\end{scriptsize}
\end{tikzpicture}
\par\end{centering}
\caption{\label{fig:CANO-label}Schematic picture of the labeling for $n=7$
and $r=1$.}
\end{figure}

The lines $\ell_{r}',\,r\in\ZZ/n\ZZ$ are the $n$ lines of symmetries
of $\cC_{0}$: the line arrangement $\cC_{1}=(\ell_{1}',\dots,\ell_{n}')$
is the union of the lines passing though the center and one vertex
of the regular $n$-gon.

The labelled line arrangement $\cC_{0}\cup\cC_{1}$ of $2n$ lines
is known as the regular line arrangement $\cA(2n)$; it is a simplicial
line arrangement.
\begin{rem}
From the symmetries of the polygon, the $\frac{n(n-1)}{2}$ double
points of $\cC_{0}$ are the $\frac{n(n-1)}{2}$ triple points of
$\cC_{0}\cup\cC_{1}$.
\end{rem}

\begin{defn}
\label{def:Matroid-Mk}Let $M_{n}$ denote the matroid obtained from
the labeled arrangement $\cC_{0}\cup\cC_{1}$ by removing, from the
matroid $M(\cC_{0}\cup\cC_{1})$, the non-bases associated with the
central singularity. We denote by $\cR_{n}$ the realization space
of $M_{n}$. 
\end{defn}

For example, the matroid $M_{7}$ can be obtained from Figure \ref{fig:MATROID}.
Geometrically, by construction, a realization of $M_{n}$ is a deformation
of the union $\cC_{0}\cup\cC_{1}$ of the regular $n$-gon and its
lines of symmetry, that preserves the incidences between the $2n$
lines, except for those at the central point of symmetry. As a result,
the central point is replaced by $\frac{n(n-1)}{2}$ double points.

For $n=2k+1\ge7$, one has $\Lambda_{\{2\},\{k\}}(\cC_{0})=\cC_{1}$.
Let us fix $c\in\ZZ/n\ZZ$. We remark that the union $U$ of the pairs
$\{i,j\},\,i\neq j$ such that $i+j+c=0$ satisfies the relation $U\cup\{-\tfrac{c}{2}\}=\ZZ/n\ZZ$,
where $-\tfrac{c}{2}\notin U$. This implies that the $k$ points
$p_{i,j}\,\text{with }\,i+j+c=0\,\mod n$ are also collinear double
points of the line arrangement $\cC_{0}\setminus\{\ell_{-\tfrac{c}{2}}\}=\sum_{i\neq-\tfrac{c}{2}}\ell_{i}$.
That also implies that the number of double points of $\cC_{0}\setminus\{\ell_{-\tfrac{c}{2}}\}$
on the lines $\ell'_{a}$ with $a\neq c$ is $k-1$. Therefore, one
has the equality 
\[
\Lambda_{\{2\},\{k\}}(\cC_{0}\setminus\{\ell_{-\tfrac{c}{2}}\})=\ell_{c}',
\]
so that we may consider $\Lambda_{\{2\},\{k\}}$ as an operator acting
on labelled line arrangements as follows: 

For any realization $\cC_{0}'\cup\cC_{1}'$ of $M_{n}$, since the
incidences between the lines in $\cC_{0}'$ and the lines $\cC_{1}'$
are the same as for the lines in $\cC_{0}$ and $\cC_{1}$ of the
regular $n$-gon, (except for the central singularity, but this is
not relevant), one also has, for $n\ge7$, that $\Lambda_{\{2\},\{k\}}^{\ell}(\cC_{0})=\cC_{1}$,
the $c^{th}$ line of $\cC_{1}$ is given by $\Lambda_{\{2\},\{k\}}(\cC_{0}\setminus\{\ell_{-\tfrac{c}{2}}\})$,
for $\cC_{0}=(\ell_{1},\dots,\ell_{n})$.

\begin{figure}[h]
\begin{centering}

\begin{tikzpicture}[scale=0.25]
\clip(-21.365,-15.33) rectangle (18.02,17.88);
\draw [line width=0.3mm,domain=-20.365:18.02] plot(\x,{(-10.302397111536783-1.8704694055762006*\x)/-2.34549444740409});
\draw [line width=0.3mm,domain=-20.365:18.02] plot(\x,{(-7.956902664132689-1.8704694055762001*\x)/2.345494447404089});
\draw [line width=0.3mm,domain=-20.365:18.02] plot(\x,{(-7.958578673150907--0.6675628018689426*\x)/2.9247837365454705});
\draw [line width=0.3mm,domain=-20.365:18.02] plot(\x,{(--10.883362409696378-0.6675628018689439*\x)/2.9247837365454705});
\draw [line width=0.3mm] (-3.,-15.33) -- (-3.,17.88);
\draw [line width=0.3mm,domain=-20.365:18.02] plot(\x,{(-9.00376595807958--2.702906603707257*\x)/1.3016512173526746});
\draw [line width=0.3mm,domain=-20.365:18.02] plot(\x,{(-10.305417175432254--2.702906603707257*\x)/-1.3016512173526738});
\draw [line width=0.3mm,color=blue,domain=-20.365:18.02] plot(\x,{(-4.226434953898144-0.*\x)/-8.452869907796291});
\draw [line width=0.3mm,color=blue,domain=-20.365:18.02] plot(\x,{(--9.502155104945505-8.240938811152406*\x)/17.11248576920136});
\draw [line width=0.3mm,color=blue,domain=-20.365:18.02] plot(\x,{(--2.3522804693791564-10.276282612990718*\x)/2.34549444740409});
\draw [line width=0.3mm,color=blue,domain=-20.365:18.02] plot(\x,{(-2.340052480306287-8.240938811152398*\x)/-6.571929401302233});
\draw [line width=0.3mm,color=blue,domain=-20.365:18.02] plot(\x,{(-4.2318769209959335--8.240938811152395*\x)/-6.571929401302231});
\draw [line width=0.3mm,color=blue,domain=-20.365:18.02] plot(\x,{(-0.006786021975052847--10.27628261299071*\x)/2.345494447404089});
\draw [line width=0.3mm,color=blue,domain=-20.365:18.02] plot(\x,{(--4.223414890002658--4.573376009283455*\x)/9.496713137847706});
\begin{scriptsize}
\draw [fill=white] (0.15,0.5) circle (7mm);
\draw [fill=black] (-3.,2.) circle (2.6mm);
\draw[color=black] (-3.8,2.99) node {$p_{23}$};

\draw [fill=black] (-3.,-1.) circle (2.6mm);
\draw[color=black] (-2.41,-0.1875) node {$p_{34}$};

\draw [fill=black] (-0.6545055525959113,-2.8704694055762) circle (2.6mm);
\draw[color=black] (-0.07,-2.0325) node {$p_{45}$};

\draw [fill=black] (2.2702781839495594,-2.2029066037072575) circle (2.6mm);
\draw[color=black] (2.855,-1.3575) node {$p_{56}$};

\draw [fill=black] (3.571929401302234,0.5) circle (2.6mm);
\draw[color=black] (4.99,1.0425) node {$p_{67}$};

\draw [fill=black] (2.2702781839495603,3.2029066037072567) circle (2.6mm);
\draw[color=black] (2.855,4.0425) node {$p_{17}$};

\draw [fill=black] (-0.65450555259591,3.8704694055762006) circle (2.6mm);
\draw[color=black] (-0.07,4.7175) node {$p_{12}$};
\draw[color=black] (-18,-10.99) node {$\ell_2$};
\draw[color=black] (14.35,-13.5) node {$\ell_4$};
\draw[color=black] (17,2.1325) node {$\ell_5$};
\draw[color=black] (-19.735,7.5) node {$\ell_1$};
\draw[color=black] (-2.1,-14.2) node {$\ell_3$};
\draw[color=black] (10.6,17.1825) node {$\ell_6$};
\draw[color=black] (-5.3,17.1825) node {$\ell_7$};

\draw [fill=black] (-4.880940506494056,0.5) circle (2.5mm);
\draw[color=black] (-4.9,1.5525) node {$p_{24}$};

\draw[color=blue] (-19.6,1.2525) node {$\ell_1'$};

\draw [fill=black] (-12.49671313784771,6.573376009283462) circle (2.5mm);
\draw[color=black] (-11.905,7.3275) node {$p_{14}$};

\draw [fill=black] (-12.496713137847706,-5.573376009283455) circle (2.5mm);
\draw[color=black] (-12.905,-4.8225) node {$p_{25}$};

\draw [fill=black] (-3.,-13.146752018566911) circle (2.5mm);
\draw[color=black] (-4.3,-12.9) node {$p_{36}$};

\draw [fill=black] (8.84220758525179,-10.443845414859652) circle (2.5mm);
\draw[color=black] (9.925,-9.6825) node {$p_{47}$};

\draw [fill=black] (14.112485769201351,0.5) circle (2.5mm);
\draw[color=black] (14.69,1.2525) node {$p_{15}$};

\draw [fill=black] (8.842207585251794,11.443845414859656) circle (2.5mm);
\draw[color=black] (8.125,12.1875) node {$p_{26}$};

\draw [fill=black] (-3.,14.146752018566918) circle (2.5mm);
\draw[color=black] (-1.71,14.1875) node {$p_{37}$};

\draw [fill=black] (-3.,4.405813207414516) circle (2.5mm);
\draw[color=black] (-2.41,5.1675) node {$p_{13}$};

\draw [fill=black] (-3.,-3.4058132074145138) circle (2.5mm);
\draw[color=black] (-3.99,-2.9625) node {$p_{35}$};

\draw [fill=black] (1.2264349538981438,-4.3704694055762) circle (2.5mm);
\draw[color=black] (0.0,-4.275) node {$p_{46}$};

\draw [fill=black] (4.615772631353649,2.667562801868942) circle (2.5mm);
\draw[color=black] (6.195,2.925) node {$p_{16}$};

\draw [fill=black] (1.2264349538981456,5.3704694055762) circle (2.5mm);
\draw[color=black] (1.8,6.6125) node {$p_{27}$};

\draw [fill=black] (4.615772631353649,-1.667562801868944) circle (2.5mm);
\draw[color=black] (5.195,-0.9075) node {$p_{57}$};

\draw[color=blue] (15.,-8.) node {$\ell_2'$};
\draw[color=blue] (2.04,-13.1825) node {$\ell_4'$};
\draw[color=blue] (-12.5,-14) node {$\ell_6'$};
\draw[color=blue] (-12.,17.1825) node {$\ell_3'$};
\draw[color=blue] (3.,17.1825) node {$\ell_5'$};
\draw[color=blue] (14.,8.2) node {$\ell_7'$};
\end{scriptsize}
\end{tikzpicture}
\par\end{centering}

\caption{\label{fig:MATROID}A line arrangement (almost) realizing the matroid
$M_{7}$.}
\end{figure}
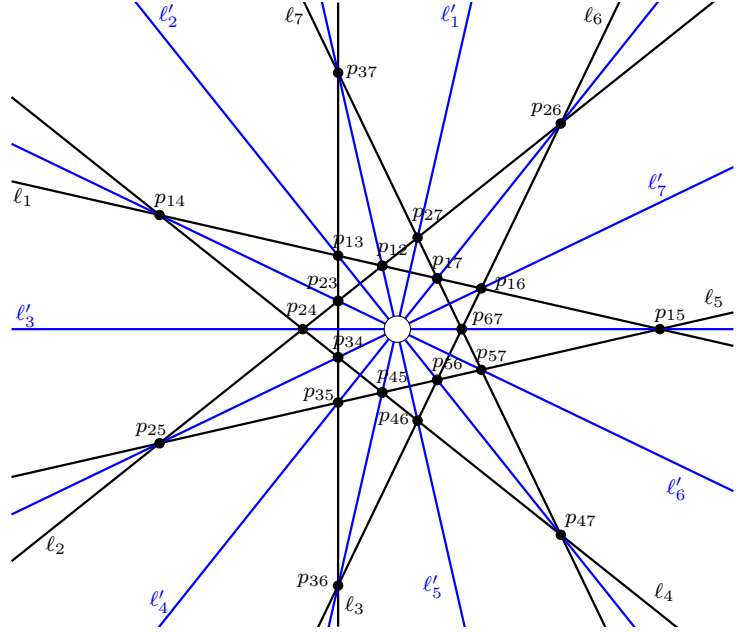

\subsubsection{Even number of sides}

Let $n=2k\geq6$ be an even integer. Let $\cC_{0}=(\ell_{1},\dots,\ell_{n})$
be the union of the lines forming a regular $n$-gon. Similar to the
case where $n$ is odd, we label the lines $\ell_{j}$ in a anti-clockwise
direction and consider the index $j$ of $\ell_{j}$ in $\ZZ/n\ZZ$.
Let $\cC_{1}$ denote the $n$ lines of symmetry. The line arrangement
$\cC_{0}$ has $\tfrac{n(n-1)}{2}$ double points, and the lines of
symmetry contain either $k$ or $k-1$ double points of $\cC_{0}$. 

For two lines $\ell_{i}\neq\ell_{j}$, let $p_{i,j}$ be the intersection
point of $\ell_{i}$ and $\ell_{j}$. For $r\in\ZZ/n\ZZ$, we define
the line $\ell'_{r}$ of $\cC_{1}=(\ell_{1}',\dots,\ell_{n}')$ as
the line containing the points $p_{i,j}$ such that $i\neq j$ and
$i+j+r=0\mod n$. There are $n/2-1$ (respectively $k=n/2$) such
points if $r$ is even (respectively odd). 

As in Definition \ref{def:Matroid-Mk}, we define the matroid $M_{n}$
to be the matroid~$M(\cC_{0}\cup\cC_{1})$, where the non-bases from
the central intersection point are removed. Analogously, we define
$\cR_{n}$ to be the realization space of the matroid $M_{n}$.

For $n\ge8$, one has $\L_{\{2\},\{k-1,k\}}^{u}(\cC_{0})=\cC_{1}$.
The labeling of the lines of $\cC_{1}$ as described above allows
us to define a labeled operator such that $\L_{\{2\},\{k-1,k\}}^{\ell}(\cC_{0})=\cC_{1}$
as follows: The operator $\L_{\{2\},\{k-1,k\}}^{\ell}$ associates
to a labeled line arrangement of $n$ lines, the union of the (possibly
empty) set of lines $\ell_{r}'$ such that $\ell_{r}'$ contains the
points $p_{i,j}$ with $i\neq j$ and $i+j+r=0\mod n$, where $p_{i,j}=\ell_{i}\cap\ell_{j}$.

\subsubsection{Explicit description of $M_{n}$}

One may define the matroid $M_{n}$ as follows: its set of atoms
is the disjoint union 
$E=\ZZ/n\ZZ\mathring{\cup}\widetilde{\ZZ/n\ZZ}$ where
$\widetilde{\ZZ/n\ZZ}$ is a disjoint copy of $\ZZ/n\ZZ$. 
The non-bases are the triples
$\{i,j;r\}\subset E$ such that 
$i,j\in\ZZ/n\ZZ$ and $r\in\widetilde{\ZZ/n\ZZ}$
with 
\[
i+j+r = 0 \text{ in }\ZZ/n\ZZ.
\]
For $a\in(\ZZ/n\ZZ)^{*}$ and $b\in\ZZ/n\ZZ$, the triple $\{i,j;r\}$
is a non-basis if and only if $\{ai+b,aj+b;ar-2b\}$ is a non-basis.
Therefore, the group $\aut(M_{n})$ of automorphisms of $M_{n}$ (i.e.
the group of bijections of $E$ that preserve the set of non-bases)
contains the group $\ZZ/n\ZZ\rtimes(\ZZ/n\ZZ)^{*}$ of invertible
affine transformations of $\ZZ/n\ZZ$.
It is a simple but lengthy
exercise to check that in fact $\ZZ/n\ZZ\rtimes(\ZZ/n\ZZ)^{*}=\aut(M_{n})$,
we omit the proof as we will not use it. But we observe that there
is a group $\ZZ/n\ZZ\rtimes(\ZZ/n\ZZ)^{*}$ acting on the surface $\Xi_1(n)$ where
$\ZZ/n\ZZ$ acts by the translation via the $n$-torsion sections,
and $a\in(\ZZ/n\ZZ)^{*}$ acts through the map $(E,t,p)\to(E,at,p)$.
In the cases of $n=7$ and $8$, the action of $\aut(M_{n})$ on $\Xi_{1}(n)$ is faithful, see \cite{KR2}.

\begin{rem}
The realization space $\cR(\mathcal{T}_{n})$
of the matroid $\mathcal{T}_n$ whose ground set is $\ZZ/n\ZZ$, and the non-bases are
triples $\{i,j,r\}\subset\ZZ/n\ZZ$ such that $i+j+r=0$ is studied in~\cite{RoulleauCurves}.
For $n\geq10$,
it is shown that $\cR(\mathcal{T}_{n})$ is an open sub-scheme of the modular
curve $X_{1}(n)$. 
\end{rem}

\subsection{\label{subsec:A-point-realization}A point realization of $M_{n}$
is on a unique cubic curve}

Let $\mathbb{K}$ be a field of characteristic $\ell\geq0$. For $n\geq7$,
let 
\[
\cP=(p_{i})_{i\in\ZZ/n\ZZ}\cup(q_{i})_{i\in\ZZ/n\ZZ}
\]
be a point arrangement which is a realization of $M_{n}$ over the
field $\mathbb{K}$.  The aim of this section is to prove the following result:
\begin{thm}
\label{Thm:ExiststUniCubCur}There exists a unique cubic curve containing
the realization~$\cP$.
\end{thm}

Let $i_{0},j_{0},k_{0}$ be integers such that $i_{0}+j_{0}+k_{0}=0$.
Let 
\[
u_{i_{0}-1},u_{i_{0}},u_{i_{0}+1};\,v_{j_{0}-1},v_{j_{0}},v_{j_{0}+1};\,w_{k_{0}-1},w_{k_{0}},w_{k_{0}+1},
\]
be $9$ distinct points on the projective plane. Assume that for all
indices $i,j,k$ such that $i+j+k=0$, the points $u_{i},v_{j},w_{k}$
are collinear. For the proof of Theorem \ref{Thm:ExiststUniCubCur},
we will need the following result, stated in \cite[Lemma 4.3]{GT}:
\begin{lem}
\label{lem:CHASLES} Any cubic curve passing through eight of these
points must also pass through the ninth point.
\end{lem}

Recall that Chasles's Theorem (\cite[Theorem CB3]{EGH}) states that a cubic curve containing
$8$ points among the $9$ intersection points of two cubic curves
necessarily contains the ninth point. Lemma \ref{lem:CHASLES} is
an application of Chasles's Theorem to the cubic curves $\ell_{1,0,-1}+\ell_{0,-1,1}+\ell_{-1,1,0}$
and $\ell_{0,1,-1}+\ell_{1,-1,0}+\ell_{-1,0,1}$, where $\ell_{i,j,k}$
denotes the line passing through the points $u_{i},v_{j},w_{k}$ such
that $i+j+k=0$. Chasles's Theorem holds over any field, see e.g.
\cite[Introduction]{EGH}. 
\begin{proof}[Proof of Theorem \ref{Thm:ExiststUniCubCur}] Let $\{i,j\}\subset\ZZ/n\ZZ$
be a subset of order two. By construction of the matroid $M_{n}$,
the realization $\cP$ is such that there exists a line $l_{i,j,k}$
containing $p_{i},p_{j}$ and $p_{k}$ if and only if $i+j+k=0$.
Moreover, when this condition is satisfied, the points $p_{i},p_{j},p_{k}$
are the only points of $\cP$ lying on the line $l_{i,j,k}$. 

Let $k'$ be the integer such that $n=2k'+1$ or $n=2k'+2$, depending
on the case. Define the sets

\[
I=\{-k',\dots,k'\},\,J=\{-k',\dots,-1\},\,K=\{1,\dots,k'\},
\]
and define the points $(u_{i})_{i\in I},\,(v_{j})_{j\in J}$ and $(w_{k})_{k\in K}$
as follows:
\[
\begin{array}{c}
u_{i}=q_{i},\,i\in I,\,\,\,v_{j}=p_{j},\,j\in J,\,\,\,w_{k}=p_{j},\,k\in K.\end{array}
\]
For $i\in I,\,j\in J,\,k\in K$, the points $u_{i},v_{j}$ and $w_{k}$
are collinear if and only if $i+j+k=0$. We can therefore apply \cite[Lemma 4.4]{GT}
to conclude that the points $q_{i},\,i\in\{-k',\dots,k'\}$ and $p_{j},\,j\in\{-k',\dots,k'\}\setminus\{0\}$
lie on a unique cubic curve $\g$. Lemma 4.4 of \cite{GT} is derived
by repeatedly using Lemma \ref{lem:CHASLES}. Although it is stated
for point arrangements over $\RR$, we have verified that the proof
holds over arbitrary fields.

If $n$ is odd (respectively, even), it remains to show that the
cubic $\g$ contains $p_{0}$ (respectively, $p_{0},p_{k}$ and $q_{k}$).
Lemma \ref{lem:CHASLES} can be used to prove that these points belongs
to the cubic $\g$. For example, the cubic $\g$ contains the $8$
points
\[
\text{\ensuremath{\cP}}_{8}=\{p_{-1},p_{1},q_{1},q_{2},q_{3},q_{-3},q_{-2},q_{-1}\}.
\]
Since the $9$ points $p_{-1},p_{0},p_{1};q_{1},q_{2},q_{3};q_{-3},q_{-2},q_{-1}$
satisfy the hypotheses of Lemma \ref{lem:CHASLES}, the cubic $\g$
must contain the point $p_{0}$. The remaining cases are similar,
and we leave their proofs to the reader.
\end{proof}

\subsection{\label{subsec:Labelled-arrangements-of}Arrangements of translates
of a point by torsion points}

In this section, we work over an algebraically closed field $\mathbb{K}$,
with no assumptions on its characteristic $\ell$. Let $E\hookrightarrow\PP^{2}$
be an elliptic curve over $\kk$, with neutral element $O$. Recall
that if $\ell=0$ or if $q$ is coprime to $\ell>0$, the group of
$q$-torsion points is $E[q]\simeq(\ZZ/q\ZZ)^{2}$, and if $\ell>0$,
the group $E[\ell^{m}]$ of $\ell^{m}$-torsion points of $E$ is
either trivial or $E[\ell^{m}]\simeq\ZZ/\ell^{m}\ZZ$ (which is the
case for $E$ generic).

We thus make the following hypothesis on the elliptic curve $E$:
there exists a cyclic sub-group $T_{O}$ of $E$ of order $n>1$. 

For $t\in T_{O}$ and a point $p$ of $E$, let us denote by $p_{t}$
the translate $p_{t}=p+t$. We define the labelled arrangement $T_{p}$
as
\[
T_{p}=(p_{t})_{t\in T_{O}}=(p+t)_{t\in T_{O}}.
\]
Recall that $\cD(\cP)$ denotes the line arrangement dual to a point
arrangement $\cP$. If $n=2k+1$ is odd, we define
\begin{equation}\label{eq:lambda}
\L = \begin{cases}
	\L_{\{2\},\{k\}} & \text{if $n=2k+1$ is odd and}\\
	 \L_{\{2\},\{k-1,k\}} &\text{if $n=2k$ is even} . 
\end{cases}
\end{equation}
\begin{thm}
\label{thm:T-mois-deux}Suppose that $n\geq7$ and assume that $6p\notin T_{O}$.
Then $\L(\cD(T_{p}))=\cD(T_{-2p})$ and the union $\cD(T_{p})\cup\cD(T_{-2p})$
is a realization of the matroid $M_{n}$. 
\end{thm}

\begin{rem}
\begin{enumerate}[a)]
\item The condition $6p\notin T_{O}$ is necessary, as the proof will
show.
\item The line arrangements $\cD(T_{p}),\cD(T_{-2p})$ are labeled by
$T_{O}$. However, choosing any isomorphism $T_{O}\simeq\ZZ/n\ZZ$
(which corresponds to the choice of a generator for $T_{O}$) provides
a labeling by $\ZZ/n\ZZ$. This justifies the claim that
$\cD(T_{p})\cup\cD(T_{-2p})$ is a realization of $M_{n}$ as the ground set of this matroid is $\ZZ/n\ZZ\mathring{\cup}\ZZ/n\ZZ'$.
\item Instead of a smooth cubic curve, one can also consider the complement~$E$ of the node of a nodal cubic. Then $E(\mathbb{K})$ is isomorphic
to $\mathbb{K}^{*}$ and its $n$-torsion points are the $n$-th roots
of unity.
That also leads to realizations of $M_{n}$, see Section
\ref{subsec:The-nodal-cubic}. See also Section \ref{subsec:Other-singular-cubic}
for realizations of $M_{\ell}$ using the cuspidal cubic in characteristic
$\ell$.
\end{enumerate}
\end{rem}

\begin{proof}
For $t\in T_{O}$, let $\ell_{t}$ denote the line dual to the point
$p+t$. Let $t,t',t''$ be three distinct elements of $T_{O}$. Suppose
that the lines $\ell_{t},\ell_{t'},\ell_{t''}$ meet at a common point.
The line dual to that point would then contain the points $p_{t},p_{t'},p_{t''}$,
which implies $p_{t}+p_{t'}+p_{t''}=O$ in $E$. This leads to the
relation $3p=-(t+t'+t'')$. That contradicts the assumption that $3p\notin T_{O}$.
Therefore, the line arrangement $\cD(T_{p})$ contains only double
points. By the same reasoning, $\cD(T_{-2p})$ has also only double
points, since $6p\notin T_{O}$ by assumption.

Let $p_{t,t'}$ denote the intersection point of $\ell_{t}$ and
$\ell_{t'}$; the dual of $p_{t,t}$ is the line $\ell_{t,t'}$, which
contains the points $p_{t},p_{t'}$. This line intersects the cubic
$E$ at a third point, namely the point $-(p_{t}+p_{t'})=-2p-t-t'\in T_{-2p}$,
which does not belong to $T_{p}$ since $3p\notin T_{O}$.

\medskip{}
Fix an element $t_{o}\in T_{O}$ and let $t,t'\in T_{O}$ with $t\neq t'$.
The line $\cD(-2p+t_{o})$ contains the double point $p_{t,t'}$ if
and only if the line $\ell_{t,t'}$ contains the points $-2p+t_{o},p+t$,
and $p+t'$. By the geometry of the cubic curve $E$, this is
equivalent to 
\[
(-2p+t_{o})+(p+t)+(p+t')=O,
\]
which is equivalent to $t+t'+t_{o}=O$. 

\medskip{}
From these descriptions of the line arrangements $\cD(T_{p})\cup\cD(T_{-2p})$,
and by taking a generator $t$ of $T_{O}$, which induces an isomorphism
$\ZZ/n\ZZ\to T_{O},k\mapsto kt$, one obtains that $\cD(T_{p})\cup\cD(T_{-2p})$
is a realization of $M_{n}$. Moreover, from the discussion in Section
\ref{subsec:Description-Matro}, one has that $\L(\cD(T_{p}))=\cD(T_{-2p})$.
\end{proof}
\begin{rem}
The results presented in this and the following subsection require
$n\ge7$, primarily because for $n=5,6$, the $n$-polygon and its
line of symmetries have not enough $k$-points for $k\geq3$.
Separate combinatorial constructions of different operators are described in
Sections \ref{sec:The-pentagon} and \ref{sec:The-hexagon}, which leads to a generalization of the presented results to the cases $n=5,6$.
\end{rem}

For $n=2k+1\geq7$ odd, let $\Psi=\Psi_{\{2\},\{k\}}$; for $n=2k\geq8$
even, let $\Psi=\Psi_{\{2\},\{k-1,k\}}$. Since by Theorem \ref{thm:T-mois-deux},
one has $\Psi(T_{p})=T_{-2p},$ the operator $\Psi$ provides a geometric
method to compute the multiplication by $-2$ (and its powers) on
a point $p$ of an elliptic curve $E$ without requiring the computation
of a tangent to the curve, or even the knowledge of its equation.
However, this comes at the price of needing to know the points in
$T_{p}$. Naturally, the curve $E$ can be reconstructed from the
knowledge of $T_{p}$ by determining the unique cubic curve passing
through $T_{p}\cup\Psi(T_{p})$.

For the cases $n=5$ and $n=6$, developed in the Sections \ref{sec:The-pentagon},
\ref{sec:The-hexagon} and in \cite{KR2} for $n=7$ and $n=8$, the
labelled point arrangements $T_{p}$ are constructed without requiring
knowledge of the curve $E$ containing $p$.

\subsection{\label{subsec:Converse} Realization spaces $\protect\cR_{n}$ and
the modular surfaces $\Xi_{1}(n)$}

In this section, we work over an algebraically closed field $\mathbb{K}$ of characteristic $0$ (see also Remark \ref{PosCar} for the positive characteristic). 

Consider the map which to a triple $(E,t,p)$ -- where $E$ is an elliptic
curve, $t$ is a generator of a cyclic group of order $n$ and $p$
a point on $E$ -- associates the labeled point arrangement $(p+kt)_{k\in\ZZ/n\ZZ}$,
considered up to projective transformations. 
\begin{prop}
\label{prop:injXitoRn}Suppose $n\geq7$.
The map $(E,t,p)\mapsto \cD(p+kt)_{k\in\ZZ/n\ZZ}\cup \cD(-2p+kt)_{k\in\ZZ/n\ZZ} $ defines a rational map
\[
\G:\Xi_{1}(n)\to\cR_{n/\mathbb{K}}
\]
which is generically nine-to-one onto its image. The fiber over the
 line arrangement $\Gamma(E,t,p)$ consists of the nine points $(E,t,p+t_{3})$,
where $t_{3}$ is in $E[3]$, the set of $3$-torsion points of $E$.
\\
The map $\G$ induces a birational map between $\Xi_{1}(n)$ and the
image of $\G$ in $\cR_{n/\mathbb{K}}$. 
\end{prop}

\begin{rem}
In Section \ref{subsec:The-nodal-cubic}, we discuss the case of the
nodal cubic curve, so that the genericity assumption in Proposition
\ref{prop:injXitoRn} can be made more precise as follows: The map
$\G$ is well-defined for any point $(E,t,p)$ in $\Xi_{1}(n)$, where
$E$ is either a smooth or a nodal cubic, and $p$ is not a $6n$-torsion
point, allowing one to apply Theorem \ref{thm:T-mois-deux}. 
\end{rem}

\begin{proof}[Proof of Proposition \ref{prop:injXitoRn}]
Let $x=\Gamma(E,t,p)$ be
a point of the image of $\Gamma$, where $E$ is smooth and $6p\not\notin T_{O}=\left\langle t\right\rangle $.
Let $(p+kt)_{k\in\ZZ/n\ZZ}$ be the corresponding point arrangement in $\PP^{2}$
(it is well defined up to projective transformation). 
By Theorem~\ref{thm:T-mois-deux}, the line arrangement $x$ is a realization of the matroid~$M_n$, hence $x\in \cR_n$.

Let us now prove that the map $\Gamma$ is indeed a rational map, i.e., the map is algebraic.
The map $\Gamma$  is defined on the dense subset of $\Xi_{1}(n)$ parametrized by the triples $(E,p,t)$ as above.
Since $E$ is a elliptic curve in $\PP^2_\mathbb{K}$, we can assume it is defined by the Weiserstrass equation 
\begin{equation}\label{eq:weierstrass}
	y^2 = x^3+ax+b,
\end{equation}
with parameters $a,b \in \mathbb{K}$ such that $4a^3+27b^2 \neq 0$.
The point $p$ is a general point $(x_1,y_1)\in \mathbb{K}^2$ satisfying the Equation~\eqref{eq:weierstrass}.
Finally, $t$ is an $n$-torsion point $(x_2,y_2)\in \mathbb{K}^2$ satisfying the Equation~\eqref{eq:weierstrass} and the $n$-th division polynomial of the elliptic curve (which defines the $n$-torsion points on $E$ and also depends on $a$ and $b$).
Given these parameters we can obtain the matrix defining the realization $\Gamma(E,p,t)$ as a polynomial map in $x_1,x_2,y_1,y_2$.
The above argument shows that this matrix is indeed in $\cR_n$ and the map $\Gamma$ is hence algebraic.

\medskip
We denote by $\mathcal{U}(x)$ the union of the points
$(p+kt)_{k\in\ZZ/n\ZZ}$ and $(-2p+kt)_{k\in\ZZ/n\ZZ}$, the latter is the the image of the former arrangement under the operator $\Psi$.
By Theorem \ref{Thm:ExiststUniCubCur},
there is a unique cubic curve passing through the points $\mathcal{U}(x)$. 
This cubic curve is isomorphic to $E$, and thus we identify it with
$E$. Since $\mathcal{U}(x)$ contains at least $10$ points, and
by B\'ezout's Theorem two cubic curves meet in at most $9$ distinct
points, $E$ is the unique cubic curve containing~$\mathcal{U}(x)$. 

For $p,q$ in $E$, suppose that there exists a projective transformation
$\g$ of the plane that maps the labeled point arrangement $(p+kt)_{k\in\ZZ/n\ZZ}$
to $(q+kt)_{k\in\ZZ/n\ZZ}$. Necessarily, since two distinct cubic
curves meet in at most $9$ points, $\g$ must induce an automorphism
of the projective curve $E$. 

Suppose that $E$ has $j$-invariant is different from
$0$ and $1728$. Then the group of projective transformations of
$\PP^{2}$ preserving $E$ has order $18$ and is generated by the
maps inducing the multiplication map $[-1]$ and the translations
by order $3$ torsion elements (see e.g. \cite{BM}). Therefore, the
point $q$ must be in the orbit of $p$ under that group of order $18$.
If $\tau$ is the projective transformation inducing the translation
by a $3$-torsion element $t_{3}$, the point configuration $\G(E,t,p)$
is projectively equivalent to 
\[
\tau((p+kt)_{k\in\ZZ/n\ZZ}\cup (-2p+kt)_{k\in\ZZ/n\ZZ})=\G(E,t,p+t_{3}).
\]
The point configuration $(p+kt)_{k\in\ZZ/n\ZZ}$ is projectively equivalent
to $(-p-kt)_{k\in\ZZ/n\ZZ}$, but since $n>4$, $\Gamma(E,t,p)$ is not projectively
equivalent to $(-p+kt)_{k\in\ZZ/n\ZZ}\cup (2p+kt)_{k\in\ZZ/n\ZZ}=\G(E,t,-p)$. 

Note that if $E$ has $j$-invariant $0$ or $j=1728$, then the curve
$E$ has complex multiplication by $\mu_{3}$ or $\mu_{4}$, respectively,
where $\mu_{k}$ denotes the complex $k$-th roots of unity. By \cite{BM}
Corollary 3.10, the extra projective transformations of the plane
induce, by restriction, the automorphisms $[\zeta],\,\zeta\in\mu_{k}$,
where $[\zeta]$ is the multiplication by $\zeta$ on $E$. However,
as in the case of the automorphism $[-1]$, the point configuration
$(p+kt)_{k\in\ZZ/n\ZZ}$ is projectively equivalent to $(\zeta p+k\zeta t)_{k\in\ZZ/n\ZZ}$
but not to $(\zeta p+kt)_{k\in\ZZ/n\ZZ}$ for primitive
$\zeta\in\mu_{k}$.

Let $\mathcal{E}[3]$ be the group of $3$-torsion sections acting
on $\Xi_{1}(n)$. An element of $\mathcal{E}[3]$ acts on the generic
element $(E,t,p)$ by the translation by a $3$-torsion point $t_{3}$
of $E$: $(E,t,p)\to(E,t,p+t_{3})$. From the above discussion, the
map $\G$ satisfies $\G(\tau(x))=\G(x)$ for a generic point $x$
and $\tau\in\mathcal{E}[3]$. Thus the degree $9$ map $\G$ factors
through the degree $9$ quotient map $\pi:\Xi_{1}(n)\to\Xi_{1}(n)/\mathcal{E}[3]$. 

Consider the multiplication by $3$ map  $[3]:\Xi_1(n)\to\Xi_1(n)$.
This is a degree $9$ rational map, which we claim has the same fibers as $\pi$:
The fiber of $(E,t,3p)$ under the map $[3]$ is the set $\{(E,t,p'):3p'=3p\}$ which is the same as $\{(E,t,p+t_3):t_3\in E[3]\}$.
The latter set is the same as the fiber of $(E,t,p)$ under the map $\pi$ as claimed.
Therefore  $\Xi_{1}(n)/\mathcal{E}[3]$ is birational to 
$\Xi_{1}(n)$, and thus there is a birational map from $\Xi_{1}(n)$ 
to the image of $\G$
in $\cR_{n}$. 
\end{proof}
\begin{rem}
In an earlier version of this paper, we incorrectly asserted that
the map $\G$ was one-to-one onto its image. We are grateful to Pierre
Deligne for pointing out this mistake.
\end{rem}

\begin{rem} \label{PosCar}
The result of Proposition \ref{prop:injXitoRn} should hold true
also in positive characteristic $\ell>3$. Indeed the results of \cite{BM}
that we use in the proof generalize 
in characteristic $\ell>3$,
see e.g. the MathOverFlow discussion number 484168. 
There are issues in characteristic $3$, where $E[3]$ assumes a non-reduced scheme structure. 
Also, formally, the argument building on Equation \eqref{eq:weierstrass} does not work in characteristics $2$ and $3$. 
Finally, one can check that in positive characteristic, the notion of complex multiplication  by a complex root zeta on $E$ is well defined. 
\end{rem}

From Proposition \ref{prop:injXitoRn} and Theorem \ref{Thm:ExiststUniCubCur},
we derive the following result, which implies both Theorem~\ref{thm:MAIN1}
and~\ref{thm:Main-X}. 
\begin{thm}
The realization space $\cR_{n}$ is birational to the modular elliptic
surface $\Xi_{1}(n)$. In particular it is irreducible. 
\end{thm}

\begin{proof}
Theorem~\ref{Thm:ExiststUniCubCur} yields that the image of
$\G$ contains a dense subset of~$\cR_{n}$. By Proposition~\ref{prop:injXitoRn}
that dense subset is birational to $\Xi_{1}(n)$. 
\end{proof}
In the next Section, we examine which singular cubic curves can provide
realizations of $M_{n}$. 

\section{\label{sec:Further-constructions-and}Further constructions and results}

In this section we collect several related constructions on which
the operators $\L$ act. 

\subsection{\label{subsec:The-nodal-cubic}The nodal cubic curve}

If $E$ is a nodal cubic with node $s$, Theorem~\ref{thm:T-mois-deux}
holds true modulo the following adjustments: 

Define $E'=E\setminus\{s\}$. The choice of an inflection point $O$
of $E'$ gives a group structure on $E'$ that is (isomorphic to)
the multiplicative group $\mathbb{G}_{m}$, and such that the points
corresponding to $a,b,c$ are on a line if and only if $abc=1$ (see
e.g. \cite[Chapter II, Proposition 2.5]{Silverman}). The torsion
elements of $E'$ are the $n$-th roots of unity, and Theorem \ref{thm:T-mois-deux}
is true for $E'$, with the proof following the same steps after transitioning
from the additive to the multiplicative notation for the group law
on $E$. 

For example, one may choose $E:\{y^{2}z=x^{3}+x^{2}z\}$, which is
singular at $(0:0:1)$. Furthermore, let $(0:1:0)$ be the neutral
element and fix the isomorphism $\g:\mathbb{G}_{m}\to E'$ defined
by 
\[
\g(t)=(4t^{2}-4t\,:\,4t^{2}+4t\,:\,(t-1)^{3})
\]
with the inverse map given by 
\[
(x:y:z)\to(2x^{2}+2xy+y^{2}+2xz+2yz)/y^{2}.
\]
Let $\UU_{n}$ be the group of $n$-th roots of unity. For $t\in\GG_{m}$
such that $t^{6m}\neq1$, define $\cC_{0}=\cD(\g(t\ze))_{\ze\in\UU_{n}}$
and $\cC_{1}=\L(\cC_{0})$. From the above discussion: 
\begin{prop}
\label{prop:nodal-case}The line arrangement $\cC_{0}\cup\cC_{1}$
is a realization of $\cR_{n}$. 
\end{prop}

That yields explicit realizations of the line arrangements in $\cR_{n}$
for all $n\ge7$.
Note that the explicit realizations of $\cR_{n}$ using
an elliptic curve $E$ may be difficult to obtain for large $n$,
since it is usually difficult to construct the group of $n$-torsion
points of $E$. 
\begin{rem}
A consequence of Proposition \ref{prop:nodal-case} is that the rational
map $\G:\Xi_{1}(n)\to\cR_{n}$ defined in Proposition \ref{prop:injXitoRn}
extends to the nodal fibers of the fibration $\Xi_{1}(n)\to X_{1}(n)$.
The different nodal fibers correspond to the choice of an isomorphism
$\UU_{n}\simeq\ZZ/n\ZZ$, i.e., of the choice of a generator of $\UU_{n}$. 
\end{rem}

\subsection{\label{subsec:Other-singular-cubic}Other singular cubic curves}

In this section we work over an algebraically closed field of characteristic $0$ or $p$ with $p>3$.  The non-nodal singular reduced cubic curves $C$ are:
\begin{enumerate}
\item The cuspidal cubic.
\item The union of a line and a conic in general position.
\item The union of a conic and a tangent to one point of the conic.
\item Three lines in general position.
\item Three lines meeting at the same point.
\end{enumerate}
For each of theses cases, let $C^{\#}$ be the complement of the singular
points. According to results attributed to N\'eron, which also appear in Tate's algorithm paper,
 the curve $C^{\#}$ is isomorphic, respectively,
to the group
\[
\mathbb{G}_a,\GG_{m}\times\ZZ/2\ZZ,\mathbb{G}_a\times\ZZ/2\ZZ,\GG_{m}\times\ZZ/3\ZZ,\mathbb{G}_a\times\ZZ/3\ZZ.
\]
Moreover, the neutral element for $C^{\#}$ and the above-mentioned isomorphism can be chosen such
that for three points not all on a line contained in $C$, their sum
(or product, according to the case) is the neutral element if and
only if these three points are on a line. 

Suppose $n>7$. If one of the components of $C^{\#}$ is a line (intersected
with $C^{\#}$), the set $T_{p}=(p+t)_{t\in T_{O}}$ contains at least
$1/3$ of its elements on that line. Dually, this produces a point
of multiplicity $\geq3$ on $\cD(T_{p})$, making it impossible to
obtain a realization of $M_{n}$ in this manner, as $\cD(T_{p})$
must have only double points. 

For the case of the cuspidal cubic, if the characteristic is $0$,
there are no non-trivial torsion elements on $C^{\#}\simeq\mathbb{G}_a$.
If the characteristic is $\ell>0$, every point of $C^{\#}$ is $\ell$-torsion.
Therefore,  if $T_{O}$ is a cyclic group of $C^{\#}$ of order $n\geq 7$, then $n=\ell$. 
In that case, Theorem \ref{thm:T-mois-deux}
holds for $E=C^{\#}$ (with the condition on $p$ reduced to $p\notin T_{O}$),
and the proof follows the same steps, yielding realizations of $M_{\ell}$.

\subsection{Periodic line arrangements }

In this subsection, we describe periodic line arrangements in~$\cR_{n}$
under the associated operator $\Psi$.

For an integer $n=2k+1\geq7$ (resp.\ $n=2k\geq6$), let us denote
by $\Psi$ the operator $\Psi_{\{2\},\{k\}}$ (resp.\ $\Psi_{\{2\},\{k-1,k\}}$)
and define $\L=\cD\circ\Psi\circ\cD$, the corresponding operator
acting on line arrangements. We recall that the subscript $^{u}$
means unlabeled.

As before, let $E$ be a smooth cubic curve with inflection point
$O$, and let $T_{O}$ be a cyclic subgroup of order $n$. Let $T_{p}\subseteq E$
be a subset of points in torsion progression: $T_{p}=\{p+t\,|\,t\in T_{O}\}$.

The operator $\Psi^{u}$ sends $T_{p}$ to $T_{-2p}$, therefore the
point arrangement $T_{p}$ is $\Psi^{u}$-periodic of period~$m$
if and only if $T_{p}=T_{(-2)^{m}p}$ and $T_{p}\neq T_{(-2)^{d}p}$
for $d<m$. If $T_{p}=T_{(-2)^{m}p}$, then $\exists t\in T_{O}$,
$p+t=(-2)^{n}p$ and $((-2)^{n}-1)p=t\in T_{O}$, in particular $p$
is a torsion element.

Let $p\in E[r]$ be an $r$-torsion point such that $\left\langle p\right\rangle \cap T_{O}=\{O\}$
(such a point always exists since~$T_{O}$ is cyclic and $E[r]\simeq(\ZZ/r\ZZ)^{2}$
for a complex elliptic curve). Since $\left\langle p\right\rangle \cap T_{O}=\{O\}$,
the relation $ap=t\in T_{O}$ for an integer $a$ yields $t=O$, therefore
the point arrangement $T_{p}$ is $\Psi^{u}$-periodic of period $m(r)$,
where $m(r)$ is the order of $-2$ in $(\ZZ/r\ZZ)^{*}$.

One observes that for an integer $m>2$, the element $-2$ has order
$m$ in $(\ZZ/(2^{m}-(-1)^{m})\ZZ)^{*}$. Thus, for any period $m>2$,
there exist line arrangements in $\cR_{n}$ ($n\geq7$) that are $\L^{u}$-periodic
of period $m$. 

If $(-2)^{m}-1=0\mod r,$ then $r$ divides $2^{m}-(-1)^{m}$, thus
once an elliptic curve $E$ is fixed, there is a finite number of
$m$-periodic arrangements. One may obtain periodic line arrangements
with the same period, but coming from torsion points of distinct order.
For example, $2^{12}-1=3^{2}\cdot5\cdot7\cdot13$ and the $16$ integers
$r$ such that $-2$ has order $12$ in $(\ZZ/r\ZZ)^{*}$ are 
\[
13,35,39,45,65,91,105,117,195,273,315,455,585,819,1365,4095.
\]
To each integer $r$ in that list, one may associate a line arrangement
which is $12$-periodic for the action of $\L$. Table~\ref{fig:periods}
shows for a period $k$ the number $N(k)$ of integers $r$ such that
$-2$ has period exactly $k$ in $(\ZZ/r\ZZ)^{*}$ together with the
lowest possible such number $r$. The example above is the column
with period $12$ in that table for which we listed the $16$ possible
choices for $r$.
\begin{center}
\begin{table}[hbt]
\begin{tabular}{lcccccccccccccc}
\toprule 
Period $k$  & $3$  & $4$  & $5$  & $6$  & $7$  & $8$  & $9$  & $10$  & $11$  & $12$  & $13$  & $22$  & $28$  & $60$\tabularnewline
\midrule 
$N(k)$  & $1$  & $2$  & $2$  & $3$  & $2$  & $4$  & $5$  & $4$  & $2$  & $16$  & $2$  & $12$  & $54$  & $4456$\tabularnewline
Lowest $r$  & $9$  & $5$  & $11$  & $7$  & $43$  & $17$  & $19$  & $31$  & $683$  & $13$  & $2731$  & $23$  & $29$  & $61$\tabularnewline
\bottomrule &  &  &  &  &  &  &  &  &  &  &  &  &  & \tabularnewline
\end{tabular}\caption{The choices of $r$ for various periods $k$ as explained above.}
\label{fig:periods} 
\end{table}
\par\end{center}

Using a low $r$ (and therefore a torsion sub-group with few elements)
forces the union of the line arrangements to have many triple points;
in case $n=7$ and $r=13$, the union is a line arrangement of $84$
lines with $1036$ triple points and $378$ double points. If we use
real torsion points, that is a real line arrangement. Note that by
\cite[Theorem 1.3]{GT}, the upper-bound on the number of triple points
on an arrangement of 84 real lines is $1135.$ Over any field, the
Sch\"onheim upper-bound for $84$ lines is $1148$ triple points.

\section{\label{sec:The-pentagon}The pentagon, the operator $\Lambda_{\{2\}}^{0}$
and the pentagram map}

Let us denote by $\L_{\{2\}}$ the operator $\L_{\{2\},\{2\}}$. Let
$\cC_{0}=(\ell_{1},\dots,\ell_{5})$ be a pentagon: a labelled arrangement
of $5$ lines. For $n=1,\dots,5$, each line arrangement $\Lambda_{\{2\}}(\sum_{i\neq n}\ell_{i})$
is the union of three lines and the line arrangement $\Lambda_{\{2\}}(\cC_{0})$
is the union of these $15$ lines, thus $\Lambda_{\{2\}}$ cannot
act as a self map on some realization space of line arrangement with
five lines.

Instead of using $\Lambda_{\{2\}}$, let us define combinatorially
(using Figure \ref{fig:Pentagon}) three operators $\L_{\{2\}}^{\pm},\L_{\{2\}}^{0}$
acting on labeled line arrangements of $5$ lines. These operators
are such that $\Lambda_{\{2\},\{2\}}(\cC_{0})$ is the disjoint union
of $\Lambda_{\{2\}}^{\pm}(\cC_{0})$ and $\Lambda_{\{2\}}^{0}(\cC_{0})$.

The first operator, denoted by $\L_{\{2\}}^{0}$ extends the operators
$\L_{\{2\},\{k\}}$, $k\geq3$ to the case of $5$ lines in the following
way: For a labeled pentagon $\cC_{0}=(\ell_{1},\dots,\ell_{5})$,
the labeled pentagon $\L_{\{2\}}^{0}(\cC_{0})=(\ell_{1}',\dots,\ell_{5}')$
is defined by 
\[
\begin{array}{c}
\ell_{1}'=\overline{p_{3,4}p_{2,5}},\,\ell_{2}'=\overline{p_{1,3}p_{4,5}},\,\ell_{3}'=\overline{p_{1,5}p_{2,4}},\,\ell_{4}'=\overline{p_{1,2}p_{3,5}},\,\ell_{5}'=\overline{p_{1,4}p_{2,3}},\end{array}
\]
where $\overline{pq}$ is the unique line through points $p\neq q$,
and $p_{i,j}$ is the intersection point of the lines $\ell_{i}$and
$\ell_{j}$. In the above equalities $\ell_{j}'=\overline{p_{r,s},p_{t,u}}$,
the indices are such that $\{j,r,s,t,u\}=\{1,\dots,5\}$; each of
the $10$ points $p_{i,j}$ is on a unique line $\ell_{t}'$. Let
us consider the indices as elements of $\ZZ/5\ZZ$. Then the triples
$\ell_{j}',p_{r,s},p_{t,u}$ also verify the relation 
\[
r+s=t+u=2j\,\mod5.
\]

Let us define the operator $\L_{\{2\}}^{+}$ which associates to $\cC_{0}$
the lines $\ell_{1}''$,...,$\ell_{5}''$ defined by 
\[
\begin{array}{c}
\ell_{1}''=\overline{p_{2,3}p_{4,5}},\,\ell_{2}''=\overline{p_{1,5}p_{3,4}},\,\ell_{3}''=\overline{p_{1,2}p_{4,5}},\,\ell_{4}''=\overline{p_{1,5}p_{2,3}},\,\ell_{5}''=\overline{p_{1,2}p_{3,4}}.\end{array}
\]
Moreover, let us define the operator $\L_{\{2\}}^{-}$ which associates
to $\cC_{0}$ the lines $\ell_{1}'''$,...,$\ell_{5}'''$ defined
by 
\[
\begin{array}{c}
\ell_{1}'''=\overline{p_{2,4}p_{3,5}},\,\ell_{2}'''=\overline{p_{1,4}p_{3,5}},\,\ell_{3}'''=\overline{p_{1,4}p_{2,5}},\,\ell_{4}'''=\overline{p_{1,3}p_{2,5}},\,\ell_{5}'''=\overline{p_{1,3}p_{2,4}}.\end{array}
\]
In both cases, the indices of $\ell_{j}'',p_{r,s},p_{t,u}$ (resp.\ $\ell_{j}''',p_{r,s},p_{t,u}$)
are such that $\{j,r,s,t,u\}=\{1,\dots,5\}$ and the following relation
holds: 
\[
r+s=t+u=-j\,\mod5.
\]
\begin{figure}[h]
\begin{centering}
\begin{tikzpicture}[scale=0.35]
\clip(-4.035343653759073,-11.277397700687427) rectangle (20.08514341589862,8.977421904826393);
\draw [line width=0.3mm,domain=-4.035343653759073:20.08514341589862] plot(\x,{(--10.6176-2.82*\x)/4.08});
\draw [line width=0.3mm,domain=-4.035343653759073:20.08514341589862] plot(\x,{(-27.6024--1.8*\x)/3.38});
\draw [line width=0.3mm,domain=-4.035343653759073:20.08514341589862] plot(\x,{(-37.2204--3.18*\x)/-0.9});
\draw [line width=0.3mm,domain=-4.035343653759073:20.08514341589862] plot(\x,{(-6.222-0.1*\x)/-4.9});
\draw [line width=0.3mm,domain=-4.035343653759073:20.08514341589862] plot(\x,{(--10.8188-2.06*\x)/-1.66});
\draw [line width=0.3mm,color=blue,domain=-4.035343653759073:20.08514341589862] plot(\x,{(--87.04539596231493-9.406177658142665*\x)/-1.2273082099596202});
\draw [line width=0.3mm,color=blue,domain=-4.035343653759073:20.08514341589862] plot(\x,{(--19.92886630511672-2.3058193822211748*\x)/-13.70514972883754});
\draw [line width=0.3mm,color=blue,domain=-4.035343653759073:20.08514341589862] plot(\x,{(-62.614041393583975--8.222042083477058*\x)/-7.255294929285962});
\draw [line width=0.3mm,color=blue,domain=-4.035343653759073:20.08514341589862] plot(\x,{(-102.60840120760801--10.906034014290036*\x)/13.607774982389056});
\draw [line width=0.3mm,color=blue,domain=-4.035343653759073:20.08514341589862] plot(\x,{(--19.07767541121889-2.9880134964150145*\x)/10.307338675664276});
\begin{scriptsize}
\draw [fill=black] (6.38,1.4) circle (1.5mm);
\draw[color=black] (6.794914778842196,1.9968571038724976) node {$p_{45}$};
\draw [fill=black] (4.72,-0.66) circle (1.5mm);
\draw[color=black] (5,-1.5) node {$p_{15}$};
\draw [fill=black] (8.8,-3.48) circle (1.5mm);
\draw[color=black] (9.222937318304428,-2.891135640044875) node {$p_{12}$};
\draw [fill=black] (12.18,-1.68) circle (1.5mm);
\draw[color=black] (12.609389807554383,-1.1020664004411311) node {$p_{23}$};
\draw [fill=black] (11.28,1.5) circle (1.5mm);
\draw[color=black] (11.682907522759583,2.0927000988512696) node {$p_{34}$};
\draw[color=black] (-3.588076343858135,4.536042283221506) node {$\ell_1$};
\draw[color=black] (-3.588076343858135,-9.500145288558917) node {$\ell_2$};
\draw[color=black] (8.903460668375187,8.48223309743607) node {$\ell_3$};
\draw[color=black] (-3.28076343858135,0.6) node {$\ell_4$};
\draw[color=black] (12.9,8.48223309743607) node {$\ell_5$};
\draw [fill=black] (10.027308209959621,5.926177658142665) circle (1.5mm);
\draw[color=black] (11.25,6.) node {$p_{35}$};
\draw[color=blue] (8.7,-10) node {$\ell_4'$};
\draw [fill=black] (18.42514972883754,1.6458193822211746) circle (1.5mm);
\draw[color=black] (18.43918448117458,2.1885430938300416) node {$p_{24}$};
\draw[color=blue] (-3.4922333488793624,-1.5493337103420672) node {$\ell_3'$};
\draw [fill=black] (13.635294929285962,-6.822042083477059) circle (1.5mm);
\draw[color=black] (14.547034732235966,-6.309535794287744) node {$p_{13}$};
\draw[color=blue] (0.8,8.48223309743607) node {$\ell_2'$};
\draw [fill=black] (-2.327774982389058,-9.406034014290036) circle (1.5mm);
\draw[color=black] (-1.2,-9.5) node {$p_{25}$};
\draw[color=blue] (19,7) node {$\ell_1'$};
\draw [fill=black] (1.8726613243357233,1.3080134964150145) circle (1.5mm);
\draw[color=black] (2.290294014839899,1.8371187789078776) node {$p_{14}$};
\draw[color=blue] (-3.4922333488793624,2.24) node {$\ell_5'$};
\end{scriptsize}
\end{tikzpicture}
\par\end{centering}
\caption{\label{fig:Pentagon}A pentagon arrangement and its image by $\protect\L_{\{2\}}^{0}$
in blue.}
\end{figure}
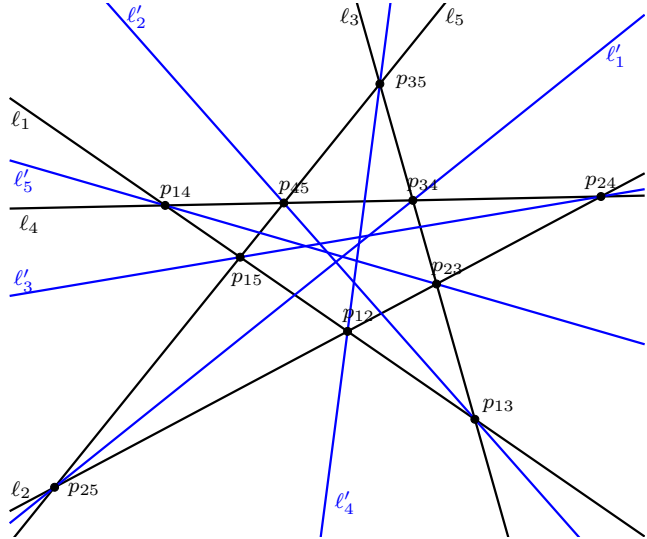

The pentagram map $\text{P}$ acts on (generic) line arrangements
$\cC=(\ell_{1},\dots,\ell_{n})$ labeled by $\ZZ/n\ZZ$, by sending
$\cC$ to the line arrangement $\text{P}(\cC)=(\ell_{1}',\dots,\ell'_{n})$,
where $\ell_{r}'$ is the line through the intersection points $\ell_{r}\cap\ell_{r+2}$
and $\ell_{r+1}\cap\ell_{r+3}$. In the case $n=5$, the operators
$\L_{\{2\}}^{+}$ and $\L_{\{2\}}^{-}$ are in fact the pentagram
map and its inverse map. For $\cC_{0}$ generic, it is known that
the pentagons $\L_{\{2\}}^{\pm}(\cC_{0})$ are projectively equivalent
to $\cC_{0}$, see \cite{SchwartzPent}, so that the pentagram map
acts trivially on the realization space of five lines.

Let $\cC_{0}(w)$ be the pentagon arrangement with normal vectors
the canonical basis and $w=(x:y:z)$. For a generic choice of $w$,
the arrangement $\L_{\{2\}}^{0}(\cC_{0}(w))$ is the pentagon arrangement
whose normal vectors are the columns of the matrix 
\[
\left(\begin{array}{ccccc}
x & y-x & z & x & 0\\
y & 0 & y & y & 1\\
z & y-z & z & 0 & 1
\end{array}\right).
\]
By sending the first four normals to the canonical basis, one obtains
that $\L_{\{2\}}^{0}$ acts on the realization space $\cR_{5}$ of
realization of $M_{5}$ though the map $\l_{\{2\}}^{0}:\PP^{2}\to\PP^{2}$
which to $w=(x:y:z)$ associates {\scriptsize{}{} 
\[
\begin{array}{c}
w'=(x^{5}y^{2}z-4x^{4}y^{3}z+5x^{3}y^{4}z-2x^{2}y^{5}z-2x^{5}yz^{2}+6x^{4}y^{2}z^{2}-2x^{3}y^{3}z^{2}-5x^{2}y^{4}z^{2}+3xy^{5}z^{2}\\
+6x^{2}y^{3}z^{3}+xy^{4}z^{3}-y^{5}z^{3}-x^{4}z^{4}+3x^{3}yz^{4}+2x^{2}y^{2}z^{4}-4xy^{3}z^{4}-x^{2}yz^{5}+y^{3}z^{5}\\
:x^{4}y^{4}-x^{3}y^{5}-5x^{4}y^{3}z+4x^{3}y^{4}z+x^{2}y^{5}z+8x^{4}y^{2}z^{2}-2x^{3}y^{3}z^{2}-6x^{2}y^{4}z^{2}-4x^{4}yz^{3}\\
-6x^{3}y^{2}z^{3}+8x^{2}y^{3}z^{3}+2xy^{4}z^{3}+4x^{3}yz^{4}+x^{2}y^{2}z^{4}-5xy^{3}z^{4}-x^{2}yz^{5}+y^{3}z^{5}\\
:x^{5}y^{2}z-4x^{4}y^{3}z+4x^{3}y^{4}z-2x^{5}yz^{2}+8x^{4}y^{2}z^{2}-6x^{3}y^{3}z^{2}-4x^{2}y^{4}z^{2}\\
+x^{4}yz^{3}-8x^{3}y^{2}z^{3}+12x^{2}y^{3}z^{3}+xy^{4}z^{3}+2x^{2}y^{2}z^{4}-6xy^{3}z^{4}+y^{3}z^{5}).
\end{array}
\]
}One gets: 
\begin{cor}
Let $\cC_{0}$ be a generic pentagon arrangement. Then $\cC_{1}=\L_{\{2\}}^{0}(\cC_{0})$
is not projectively equivalent to $\cC_{0}$. 
\end{cor}

\begin{proof}
The labeled line arrangement $\L_{\{2\}}^{0}(\cC_{0}(w))$ is projectively
equivalent to $\cC_{0}(w)$ if and only if the point $w$ is equal
to $\l_{\{2\}}^{0}(w)$. The polynomials defining $\l_{\{2\}}^{0}$
being coprime of degree $8>1$, $\l_{\{2\}}^{0}$ is not the identity
map and $w\neq\l_{\{2\}}^{0}(w)$ for a generic $w$. 
\end{proof}
The base point set of $\l_{\{2\}}^{0}$ are the eight points {\scriptsize{}{}
\[
\begin{array}{c}
(0:1:0),(1:1:1),(0:0:1),(1:0:0),(1:1:0),(1:0:1),\\
(\sqrt{5}+3:\sqrt{5}+1:2),(-\sqrt{5}+3:-\sqrt{5}+1:2).
\end{array}
\]
}There is a pencil of cubics containing these points, with base loci
the line $x=y+z$. For the two points with coordinates in $\QQ(\sqrt{5})\setminus\QQ$,
the associated pentagon $\cC_{0}$ is the regular pentagon, the arrangement
$\cC_{1}=\L_{\{2\}}^{0}(\cC_{0})$ has a unique $5$-point, which
is the center of the regular pentagon, and $\cC_{0}\cup\cC_{1}$ is
a simplicial line arrangement with $10$ lines. 
\begin{prop}
\label{prop:Degree-rational-self-map}The rational self-map $\l_{\{2\}}^{0}$
has degree $4$. 
\end{prop}

\begin{proof}
Consider a pencil $\cP$ of lines (for example $\{L_{t}:(x+ty=0)\,|\,t\in\PP^{1}\}$),
the pull-back $\{C_{t}:t\in\PP^{1}\}$ is a family of curves. One
computes that for the generic point of $\PP^{1}$, the degree of the
map $\l_{\{2\}}^{0}:C_{t}\to L_{t}$ is $4$.

\end{proof}
Let $\Psi_{\{2\}}^{0}$ be the operator acting on labelled arrangements
of $5$ points defined by $\Psi_{\{2\}}^{0}=\cD\circ\L_{\{2\}}^{0}\circ\cD$. 
\begin{thm}\label{cinq}
The realization space $\cR_{5}$ is birational to the modular elliptic
surface $\Xi_{1}(5)$. The operator $\Psi_{\{2\}}^{0}$ acts on $\Xi_{1}(5)$
as the map $(E,p,t)\to(E,[-2]p,[-2]t)$. 
\end{thm}

\begin{proof}
For the generic point $w=(a:b:1)$ in the plane, let $P_{5}=(n_{1},\dots,n_{5})$
and $P_{5}'=(n_{1}',\dots,n_{5}')$ be the normal vectors to $\cC_{0}(w)\text{ and }\cC_{1}(w)$,
where $\cC_{1}(w)=\L_{\{2\}}^{0}(\cC_{0}(w))$. One computes by using
MAGMA that there is a unique cubic curve 
\[
E_{w}:\,x^{2}y-\tfrac{a}{b}xy^{2}-ax^{2}z+\tfrac{a^{2}b-a-b^{2}+b}{b^{2}-b}xyz+\tfrac{ab-a^{2}}{b^{2}-b}y^{2}z+\tfrac{ab-a^{2}}{b-1}xz^{2}+\tfrac{a^{2}-ab}{b^{2}-b}yz^{2}
\]
which contains the normal vectors in $P_{5}\cup P_{5}'$. One computes
moreover that the curve $E_{w}$ is smooth for generic $w$, and that
the points $n_{j}-n_{1}$ and $n_{j}'-n_{1}'$ for $j\in\{1,\dots,5\}$
are $5$-torsion points on the cubic $E_{w}$. A last computation
gives that the map $\Psi_{\{2\}}^{0}=\cD\circ\L_{\{2\}}^{0}\circ\cD$
sends the labeled point arrangement $(n_{1},\dots,n_{5})$ to $(n_{1}',\dots,n_{5}')$,
and this is the map $T_{p}\to T_{-2p}$ described in Section \ref{subsec:Labelled-arrangements-of}, which is 9-to-1, with kernel the $3$-torsion points, so that the proof goes as for the cases $n>6$.  
\end{proof}

\begin{rem}
Given a word $(\e_{1},\dots,\e_{n})$, $\e_{j}\in\{-1,0,1\}$, we
can define an operator $\L_{\{2\}}^{\e_{1}}\cdots\L_{\{2\}}^{\e_{n}}$.
It would be interesting to understand if there are relations between
the operators $\L_{\{2\}}^{\e}$ other than $\L_{\{2\}}^{1}\L_{\{2\}}^{-1}=I_{d}$.
\end{rem}

\subsection*{Periodic arrangements}
A $3$-periodic line arrangement may be obtained as follows: 
Let $p$ be a $9$-torsion point on a plane elliptic curve.
With the notations as  above, define $\cC_{0}=\cD(T_{p})$, and for
$n\geq 0$ define $\cC_{n+1}=\L_{\{2\}}^{0}(\cC_n)$. 
The union  $\cC_{0} \cup \mathcal{C}_1$ is a realization of $M_5$.
By Theorem \ref{cinq}, since $(-2)^3p=p$ and therefore $T_{(-2)^3p}=T_p$, 
the sequence of line arrangements 
$\cC_n$ is $3$-periodic: $\cC_{n+3}=\cC_n$.

The line arrangement $\cC_{0}\cup\cC_{1}\cup\cC_{2}$
has singularities $t_{2}=15,t_{3}=30$. 
One computes that the realization space of the matroid associated to 
$\cC_0 \cup \cC_1 \cup \cC_2$  is a smooth irreducible curve $C$, and the map 
$\cC_0 \cup \cC_1 \cup \cC_2 \to \cC_0 \cup \cC_1 \in \cR_{5}$ is 
an embedding with an inverse, since $\cC_2=\L_{\{2\}}^{0}(\cC_1)$.
The curve $C$ has a smooth compactification $\bar{C}$ of genus $1$, which
parametrizes some line arrangements of period $3$ for $\L_{\{2\}}^{0}$.
 The $j$-invariant
of $\bar{C}$ is $-1/15$; the curve $\bar{C}$ is isomorphic to the
modular curve $X_{1}(15)$ (in the LMFDB this is the curve with label
15.a7).

We now describe further examples of periodic arrangements
under the operator $\L_{\{2\}}^{0}$ with small periods. 
\begin{enumerate}
\item For a $7$-torsion point $p$, one gets a line arrangement $\cC_{0}=\cD(T_{p})$
which is $6$-periodic. The union of the $30=6\cdot5$ lines has singularities
$t_{2}=105,t_{3}=110$.
\item For an $11$-torsion point $p$, one obtains a line arrangement $\cC_{0}=\cD(T_{p})$
which is $5$-periodic. The union of the $25$ lines has singularities
$t_{2}=150,t_{3}=50$.
\item For $13$-torsion point $p$, one gets a line arrangement which is
$12$-periodic. The union of the $60$ lines has singularities $t_{2}=210,t_{3}=520$.
The union of that line arrangement with the dual of the $5$-torsion
points is an arrangement $\cA$ of $65$ lines such that $t_{2}=64,\,t_{3}=672$.
The number of triple points of $\cA$ lines matches the upper bound
by Green--Tao \cite[Theorem 1.3]{GT} for real line arrangements.
This is explained by the fact that $\cA$ is the dual of a group of
torsion points. 
\end{enumerate}

\section{\label{sec:The-hexagon}The hexagon and the operator $\protect\L_{2|3}$}

Figure \ref{fig:Hexagon} depicts the union of the regular hexagon
and its lines of symmetries. Consider the matroid~$M_{6}$ with $12$
atoms obtained from Figure \ref{fig:Hexagon} by keeping the labeling
and removing the conditions imposed by the central point. This is
a degenerate case of the matroids defined in Section \ref{subsec:Description-Matro}:
three of the blue lines contain double points only, and the operator
$\L_{\{2\},\{2,3\}}$ would return too many lines.

For a point $p=(x:y:z)$ in an open set of $\PP^{2}$, let $\cA=\cA(p)$
be the labeled arrangement of $12$ lines defined by the following
normal vectors: the four vectors of the canonical basis of $\PP^{2}$
and (in that order) the vectors 
\begin{equation}
\begin{array}{c}
(xz:x^{2}-2xy+y^{2}+xz:yz),(xz:x^{2}-xy+xz:xy-y^{2}+yz)\\
(xz:x^{2}-2xy+y^{2}+xz:xy-y^{2}+yz),(z:x-y+z:0),(x:0:x-y+z)\\
(0:1:1),(yz:x^{2}-2xy+y^{2}+xz:yz),(x:x:y).
\end{array}\label{eq:C0-C1-of-Hexagon}
\end{equation}
With these vectors, one can compute:
\begin{prop}
\label{prop:Realisation-Hexa}For $p$ generic in $\PP^{2}$, the
line arrangements $\cA(p)$ form an open subset of the realization
space $\cR_{6}$ over $\CC$. 
\end{prop}

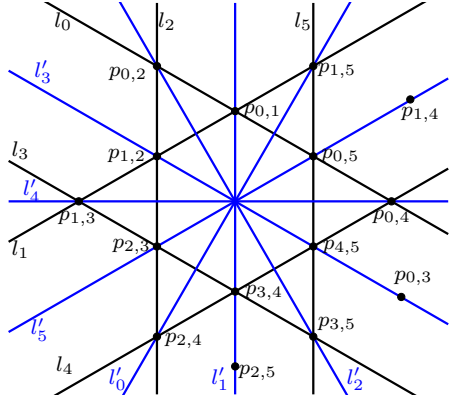
\begin{figure}[h]
\begin{centering}
\begin{tikzpicture}[scale=0.4]

 \clip(-7.5,-6.32) rectangle (7.68,6.7); \draw [line width=0.3mm,domain=-6.93:7.68] plot(\x,{(--7.15692193816531--1.5*\x)/2.5980762113533187}); \draw [line width=0.3mm] (-2.,-6.32) -- (-2.,6.7); \draw [line width=0.3mm,domain=-6.93:7.68] plot(\x,{(-6.637306695894642-1.5*\x)/2.598076211353316}); \draw [line width=0.3mm,domain=-6.93:7.68] plot(\x,{(-8.431535329954592--1.5*\x)/2.598076211353317}); \draw [line width=0.3mm] (3.196152422706633,-6.32) -- (3.196152422706633,6.7); \draw [line width=0.3mm,domain=-6.93:7.68] plot(\x,{(-8.951150572225256--1.5*\x)/-2.598076211353315}); \draw [line width=0.3mm,color=blue,domain=-6.93:7.68] plot(\x,{(--2.3138438763306124-3.*\x)/5.196152422706633}); \draw [line width=0.3mm,color=blue] (0.5980762113533172,-6.32) -- (0.5980762113533172,6.7); \draw [line width=0.3mm,color=blue,domain=-6.93:7.68] plot(\x,{(--1.2746133917892872-3.*\x)/-5.196152422706634}); \draw [line width=0.3mm,color=blue,domain=-6.93:7.68] plot(\x,{(--5.902301144450522-9.*\x)/5.196152422706632}); \draw [line width=0.3mm,color=blue,domain=-6.93:7.68] plot(\x,{(--1.0392304845413236-0.*\x)/10.392304845413264}); \draw [line width=0.3mm,color=blue,domain=-6.93:7.68] plot(\x,{(-4.863070659909189--9.*\x)/5.196152422706634}); \begin{scriptsize} \draw [fill=black] (-2.,1.6) circle (1.2mm); \draw[color=black] (-2.9945,1.58) node {$p_{1,2}$}; \draw [fill=black] (-2.,-1.4) circle (1.2mm); 

\draw[color=black] (-2.835,-1.39) node {$p_{2,3}$}; \draw [fill=black] (0.598076211353316,-2.9) circle (1.2mm);

\draw[color=black] (1.565,-2.93) node {$p_{3,4}$}; \draw [fill=black] (3.196152422706633,-1.4) circle (1.2mm);

\draw[color=black] (4.155,-1.43) node {$p_{4,5}$}; \draw [fill=black] (3.1961524227066334,1.6) circle (1.2mm);

\draw[color=black] (4.1,1.58) node {$p_{0,5}$}; \draw [fill=black] (0.5980762113533185,3.1) circle (1.2mm);

\draw[color=black] (1.505,3.1) node {$p_{0,1}$};

\draw[color=black] (-6.525,-1.49) node {$l_1$};

\draw[color=black] (-1.665,6.1) node {$l_2$}; 

\draw[color=black] (-6.525,1.81) node {$l_3$}; 

\draw[color=black] (-5.055,-5.36) node {$l_4$};

\draw[color=black] (2.865,6.1) node {$l_5$}; 
\draw[color=black] (-5.115,6.1) node {$l_0$}; 

\draw [fill=black] (-4.598076211353316,0.1) circle (1.2mm); \draw[color=black] (-4.635,-0.47) node {$p_{1,3}$}; 

\draw [fill=black] (-2.,-4.4) circle (1.2mm); 
\draw[color=black] (-1.095,-4.54) node {$p_{2,4}$}; \draw [fill=black]

(3.196152422706632,-4.4) circle (1.2mm); 
\draw[color=black] (3.99,-3.993) node {$p_{3,5}$}; \draw [fill=black]

(5.794228634059948,0.1) circle (1.2mm); 
\draw[color=black] (5.805,-0.44) node {$p_{0,4}$}; 

\draw [fill=black] (3.196152422706634,4.6) circle (1.2mm);
\draw[color=black] (3.99,4.45) node {$p_{1,5}$};

\draw [fill=black] (-2.,4.6) circle (1.2mm); 
\draw[color=black] (-2.955,4.33) node {$p_{0,2}$}; 

\draw [fill=black] (0.6,-5.39) circle (1.2mm);
\draw[color=black] (1.35,-5.59) node {$p_{2,5}$};

\draw [fill=black] (6.12,-3.08) circle (1.2mm); 
\draw[color=black] (6.525,-2.45) node {$p_{0,3}$}; 

\draw [fill=black] (6.42,3.49) circle (1.2mm); 
\draw[color=black] (6.765,3.) node {$p_{1,4}$};

\draw[color=blue] (-5.775,4.41) node {$l_{3}'$}; 

\draw[color=blue] (0.105,-5.8) node {$l_{1}'$};

\draw[color=blue] (-5.895,-4.16) node {$l_{5}'$};

\draw[color=blue] (4.61,-5.8) node {$l_{2}'$};

\draw[color=blue] (-6.15,0.47) node {$l_{4}'$};

\draw[color=blue] (-3.3,-5.8) node {$l_{0}'$}; 

\end{scriptsize}

\end{tikzpicture}
\par\end{centering}

\caption{\label{fig:Hexagon}The regular hexagon and the axes of symmetries.}
\end{figure}

{\scriptsize{}{}}Let us define combinatorially an operator $\L_{2|3}$
acting on the space of labeled hexagons. That operator is constructed
in such a way that if $\cC_{0}$ (resp.\ $\text{\ensuremath{\cC}}_{1}$)
denote the first six lines (resp.\ the last six lines) of a realization
$\cA$ of $M_{6}$, then one has $\L_{2|3}(\cC_{0})=\cC_{1}$. Let
$\cC_{0}=(\ell_{0},\dots,\ell_{5})$ be a hexagon; let us denote by
$p_{i,j}$ the intersection point of lines $\ell_{i}$ and $\ell_{j}$,
$i\neq j\in\ZZ/6\ZZ$. We define combinatorially the line arrangement
$\cC_{1}$ as follows: for each set $S_{k}$ ($0\leq k\leq5$) of
points in the following ordered list:

{\small{}{} 
\[
\{p_{1,5},p_{2,4}\},\,\,\{p_{0,1},p_{2,5},p_{3,4}\},\,\,\{p_{0,2},p_{3,5}\},\,\,\{,p_{0,3},p_{1,2},p_{4,5}\},\,\,\{p_{0,4},p_{1,3}\},\,\,\{p_{0,5},p_{1,4},p_{2,3}\},
\]
}let $\ell_{k}'$ be the union of the lines containing at least two
points of $S_{k}$. Then $\ell_{k}'$ for $k=0,2,4$ is one line and
$\ell_{1}',\ell_{3}',\ell_{5}'$ are the union of three or one line
depending on if the points in $S_{k}$ are collinear or not. The line
arrangement $\cC_{1}=\L_{2|3}(\cC_{0})$ is then the union $\ell_{0}'+\dots+\ell_{5}'$.
It contains at least $6$ lines, and if $\cC_{1}$ contains six lines
there is a natural labeling. As mentioned above, that operator is
build such that if $\cA_{1}=\cC_{0}\cup\cC_{1}$ is a generic realization
of $M_{6}$, where $\cC_{0}$ is the union of the first six lines,
then one has $\L_{2|3}(\cC_{0})=\cC_{1}$. 
\begin{thm}
\label{thm:HEXA}Let $\cA=\cC_{0}\cup\cC_{1}$ be a generic realization
of $M_{6}$ (so that $\L_{2|3}(\cC_{0})=\cC_{1}$). Then $\cC_{2}=\L_{2|3}(\cC_{1})$
is again a labeled hexagon and $\cA'=\cC_{1}\cup\cC_{2}$ is a realization
of $M_{6}$. 
\end{thm}

\begin{proof}
One computes that for $p=(x:y:z)$ generic in $\PP^{2}$, the line
arrangement $\cC_{2}=\L_{2|3}(\cC_{1})$ contains six lines, and that
the union $\cA'=\cC_{1}\cup\cC_{2}$ defines the same matroid $M_{6}$
as $\cA$. 
\end{proof}
Since $\L_{2|3}(\cC_{0})=\cC_{1}$ and $\cA=\cC_{0}\cup\cC_{1}$,
the realization space $\cR_{6}$ may also be viewed as a realization
space for the hexagons $\cC_{0}$.

Let us denote by $\l_{2|3}$ the action of $\L_{2|3}$ on the realization
space $\cR_{6}$ of $M_{6}$. By Proposition \ref{prop:Realisation-Hexa},
that action is also an action on $\PP^{2}$. One has 
\begin{prop}
The \label{prop:ational-self-map-hexa}rational self-map $\l_{2|3}:\PP^{2}\dashrightarrow\PP^{2}$
is the map which to $(x:y:z)$ associates the point{\footnotesize{}
\begin{equation}
\begin{array}{c}
(-4x^{4}yz+16x^{3}y^{2}z-28x^{2}y^{3}z+24xy^{4}z-8y^{5}z-8x^{3}yz^{2}+24x^{2}y^{2}z^{2}\\
-28xy^{3}z^{2}+12y^{4}z^{2}-5x^{2}yz^{3}+10xy^{2}z^{3}-6y^{3}z^{3}-xyz^{4}+y^{2}z^{4}\,:\,-2x^{5}y\\
+10x^{4}y^{2}-18x^{3}y^{3}+14x^{2}y^{4}-4xy^{5}-7x^{4}yz+26x^{3}y^{2}z-35x^{2}y^{3}z+20xy^{4}z\\
-4y^{5}z-9x^{3}yz^{2}+24x^{2}y^{2}z^{2}-21xy^{3}z^{2}+6y^{4}z^{2}-5x^{2}yz^{3}+9xy^{2}z^{3}\\
-4y^{3}z^{3}-xyz^{4}+y^{2}z^{4}\,:\,x^{6}-8x^{5}y+25x^{4}y^{2}-38x^{3}y^{3}+28x^{2}y^{4}-8xy^{5}\\
+3x^{5}z-19x^{4}yz+44x^{3}y^{2}z-44x^{2}y^{3}z+16xy^{4}z+3x^{4}z^{2}-15x^{3}yz^{2}\\
+24x^{2}y^{2}z^{2}-12xy^{3}z^{2}+x^{3}z^{3}-4x^{2}yz^{3}+4xy^{2}z^{3})
\end{array}\label{eq:Self-map-hexa}
\end{equation}
}The indeterminacy points of $\l_{2|3}$ are the $3$ points {\footnotesize{}$(0:0:1),(-1:0:1),(1:1:0)$.}{\footnotesize\par}

The degree of $\l_{2|3}$ is $4$. 
\end{prop}

\begin{proof}
The normal vectors $n_{5},n_{6}$ of the $5^{th}$ and $6^{th}$ line
of $\cC_{0}(p)$ (for generic $p=(x:y:z)\in\PP^{2}$) are given in
\eqref{eq:C0-C1-of-Hexagon}. Denoting $n_{k}(j),\,1\leq j\leq3,$
the $j^{\text{th}}$ coordinate of $n_{k}$, we remark that 
\[
\frac{n_{5}(1)}{n_{5}(3)}=\frac{x}{y},\,\,\frac{n_{6}(2)}{n_{6}(1)}=\frac{x^{2}-xy+xz}{xz}.
\]
Defining $a=\frac{x}{z}$, $b=\frac{y}{z}$, $u=\frac{n_{5}(1)}{n_{5}(3)},$
$v=\frac{n_{6}(2)}{n_{6}(1)}$, we get $u=\frac{a}{b}$, $v=a-b+1$,
which is equivalent to $a=u\frac{v-1}{u-1}$, $b=\frac{v-1}{u-1}$,
thus one can recover $(x:y:z)\in\PP^{2}$ from $\cC_{0}$. In other
words, the rational map $\PP^{2}\to\cR_{6}$ $p\mapsto\cC_{0}(p)$
is birational, with the inverse $\mu$ defined by $\mu(\cC_{0})=(u\frac{v-1}{u-1}:\frac{v-1}{u-1}:1)$,
where $u=\frac{n_{5}(1)}{n_{5}(3)},$ $v=\frac{n_{6}(2)}{n_{6}(1)}$.

By construction, one has $\cC_{2}=\L_{2|3}(\cC_{1})$. Let us define
$\cA_{2}=\cC_{1}\cup\cC_{2}$; it is a realization of $M_{6}$. Let
$\g\in\text{PGL}_{3}$ be the unique projective transformation such
that the first four normal vectors of $\cA_{2}$ are mapped to the
canonical basis. In order to compute the point in $\PP^{2}$ corresponding
to $\l_{2|3}(\cA_{2})$, we just have to apply $\mu$ to the line
arrangement $\g\cA_{2}$ and a computation gives the \eqref{eq:Self-map-hexa}.

For computing the degree of $\l_{2|3}$, we proceed as in the proof
of Proposition \ref{prop:Degree-rational-self-map}. 
\end{proof}
The automorphism group of $M_{6}$ is the dihedral group $D_{6}$
of order $12$, generated by permutations 
\[
s_{1}=(1,2,3,4,5,6)(7,9,11)(8,10,12),\,\,\s_{2}=(2,6)(3,5)(8,12)(9,11).
\]
One computes that the (order $6$) element $s_{1}$ acts on $\cR_{6}\subset\PP^{2}$
through the Cremona involution 
\[
s_{1}':(x:y:z)\to(-x^{2}+xy-xz,-x^{2}+2xy-y^{2}-xz+yz,yz).
\]
The element $s_{2}$ acts on $\cR_{6}\subset\PP^{2}$ through the
involution $s_{2}'(x:y:z)\to(z:x-y+z:x)$. The group generated by
$s_{1}',s_{2}'$ is the order $4$ Klein group. The self-rational
map $\l_{2|3}$ is such that 
\[
\l_{2|3}\circ s_{1}=\l_{2|3},\,\text{and }\l_{2|3}\circ s_{2}=s_{2}\circ\l_{2|3}.
\]

The pentagram map $\text{P}$ acting on arrangements of $6$ lines
$L$ is such that $\text{P}(L)$ is not projectively equivalent to
$L$, but $\text{P}^{\circ2}(L)$ is (see \cite{SchwartzPent}). One
computes that: 
\begin{prop}
The pentagram map preserves the space $\kU_{6}$ of realizations of
$M_{6}$. It acts on $\cR_{6}$ through the involution $s:(x:y:z)\to(x^{2}-xy:x^{2}-2xy+y^{2}:yz).$ 
\end{prop}

The involution $s$ is not an element of the Klein group generated
by $s_{1}',s_{2}'$; one has $s\circ s_{1}=s_{1}\circ s$ and $(s_{2}\circ s)^{2}=s_{1}$,
$(s_{2}\circ s)^{3}=s\circ s_{2}$. The involution $s$ does not preserves
the elliptic fibration of the modular surface $\Xi_{1}(6)$ since
the $j$-invariant of the elliptic curve $E$ passing through $\cD(\cC_{0}\cup\cC_{1})$
is different from the $j$-invariant of the elliptic curve $E'$ passing
through $\cD(\text{P}\cC_{0}\cup\L_{2|3}(\text{P}\cC_{0}))$. For
arrangements $\cC$ of $n\geq7$ lines, there is no $k\geq1$ such
that $\text{P}^{\circ k}(\cC)$ is projectively equivalent to $\cC$,
and we did not find other connections between the pentagram map and
the operators $\L$.


\begin{thebibliography}{10}

\bibitem{Abrashkin}  Abrashkin V., Modular representations of the Galois group of a local field, and a generalization of the Shafarevich conjecture. Math. USSR Izvestija 35 (1990), 469--518

\bibitem{BM} Bonifant A., Milnor J., On real and complex cubic curves,
Enseign. Math. 63 (2017), no. 1-2, 21--61.

\bibitem{RoulleauCurves} Borisov L., Roulleau X., Modular curves
$X_{1}(n)$ as moduli spaces of points arrangements and applications,
preprint arXiv 2404.04364

\bibitem{Magma} Bosma W., Cannon J., Playoust C., The {M}agma algebra
system. {I}. {T}he user language, J. Symbolic Comput. 24, 1997,
3--4, 235--265

\bibitem{Chai} Chai C.-L., Faltings G., Degeneration of abelian varieties. 
With an appendix by David Mumford. Erg. Math. ihr. Grenz. (3), 22. Springer-Verlag, Berlin, 1990. xii+316 pp

\bibitem{Conrad} Conrad B., Arithmetic moduli of generalized elliptic
curves, J. Inst. Math. Jussieu 6 (2007), no. 2, 209--278.

\bibitem{Oscar} Corey D., K\"uhne L., Schr\"oter B., Matroids, in {The computer algebra system OSCAR---algorithms and examples}, 351--368, Algorithms Comput. Math., 32, Springer, Cham, 2025.

\bibitem{Deligne-Rapoport} Deligne, P.; Rapoport, M. Les sch\'emas
de modules de courbes elliptiques. Modular functions of one variable
II, pp. 143--316, Lecture Notes in Math., Vol. 349, Springer, Berlin-New
York, 1973.

\bibitem{EGH} Eisenbud D., Green, M., Harris J., Cayley-Bacharach
Theorems and Conjectures, Bulletin of the AMS, Vol. 33, 3., 1996,
295--324


\bibitem{Fontaine} Fontaine J.-M., Sch\'emas propres et lisses sur $\mathbb{Z}$. In: S. Ramanan, A. Beauville (eds.), Proceedings of the Indo-French Conference on Geometry, pp. 43--56. Hindustan Book Agency, Delhi,1993.

\bibitem{GT} Green B., Tao T., On Sets Defining Few Ordinary Lines,
Discrete Comput Geom (2013) 50, 409--468

\bibitem{Katz-Mazur} Katz N., Mazur B., Arithmetic moduli of elliptic
curves, Annals of Mathematics Studies, 108. Princeton University Press,
Princeton, NJ, 1985. xiv+514

\bibitem{KR2} K\"uhne L., Roulleau X., On the dynamics of some operators
on modular elliptic surfaces $\Xi_{1}(n)$ for $n\in\{7,8\}$, Nagoya Mathematical Journal. Published online 2025:1-19. doi:10.1017/nmj.2024.35.

\bibitem{Oxley} Oxley J., Matroid Theory, second ed., Oxford Graduate
Texts in Mathematics, vol. 21, 2011. xiv+684pp.

\bibitem{OSO} Roulleau X., On some operators acting on line arrangements
and their dynamics, to appear in Enseign. Math. 

\bibitem{RoulleauUna} Roulleau X., On the dynamics of the line operator
$\Lambda_{{2},{3}}$ on some arrangements of six lines, Eur. J. of
Math. 9 (2023), no. 4, Paper No. 105, 22 pp.




\bibitem{SchwartzPent} Schwartz R.E., The pentagram map, Experiment.
Math. 1 (1992), no. 1, 71--81.

\bibitem{Shioda} Shioda T., Elliptic modular surfaces, J. Math. Soc.
Japan 24 (1972), 20--59.

\bibitem{Silverman} Silverman J., The arithmetic of elliptic curves,
GTM 106, Springer-Verlag, New York, 1992. xii+400 pp.



\end{thebibliography}
\end{document}